\documentclass[10pt, conference]{IEEEtran}
\usepackage[english]{babel} 

\usepackage[margin=54pt]{geometry}

\title{\vspace*{18pt} Distributed Frequency Control through \\ MTDC Transmission Systems}
\author{ \IEEEauthorblockA{Martin Andreasson$^{\IEEEauthorrefmark{1}}$~\IEEEmembership{Student Member,~IEEE,} Roger Wiget,~\IEEEmembership{Student Member,~IEEE,} Dimos V. Dimarogonas,\\~\IEEEmembership{Member,~IEEE,} Karl H. Johansson,~\IEEEmembership{Fellow,~IEEE} and G\"oran Andersson,~\IEEEmembership{Fellow,~IEEE}
}
\\ 
\thanks{M. Andreasson, D. V. Dimarogonas and K. H. Johansson are with the ACCESS Linnaeus Centre, KTH Royal Institute of Technology, Stockholm, Sweden. R. Wiget and G. Andersson are with the Power Systems Laboratory, ETH Zurich, Switzerland. This work was supported in part by the European Commission, the Swedish Research Council (VR) and the Knut and Alice Wallenberg Foundation.
\IEEEauthorrefmark{1} Corresponding author (mandreas@kth.se).}
} 
\IEEEoverridecommandlockouts

\usepackage[usenames,dvipsnames]{xcolor}

\usepackage{mathrsfs}
\usepackage{amsfonts}
\usepackage{amssymb}
\usepackage{amsmath}
\usepackage{graphicx}
\usepackage{amsthm}
\usepackage{commath}
\usepackage{tikz}
\usepackage{tkz-graph}
\usepackage[europeanresistors]{circuitikz}
\usetikzlibrary{arrows}
\usepackage{flushend}
\usepackage{psfrag}
\usepackage{pgfplots}
\usepackage{todonotes}
\usetikzlibrary{external}
\usepackage{blockdiagram}

\usepackage[figurename=Fig.,font=footnotesize]{caption}
\DeclareCaptionFormat{IEEEStand}{#1#2\quad #3}
\usepackage{subcaption}
\usepackage{booktabs}

\usetikzlibrary{external}
\pgfplotsset{every x tick label/.append style={font=\small, yshift=0.2ex}} 
\pgfplotsset{every y tick label/.append style={font=\small, xshift=0.3ex}} 
\pgfplotscreateplotcyclelist{Necsys}{%
	{blue,line width=1pt},
	{red,line width=1pt},
	{black!30!Red,line width=1pt},
	{cyan,line width=1pt},
	{yellow!90!black,line width=1pt},		
	{magenta,line width=1pt},
}

\usepackage{pgfpowersystems}
\usepackage{accents}

\newtheorem{theorem}{Theorem}
\newtheorem{corollary}[theorem]{Corollary}
\newtheorem{lemma}{Lemma}
\newtheorem{remark}{Remark}

\newtheorem{assumption}{Assumption}
\theoremstyle{definition}
\newtheorem{objective}{Objective}

\DeclareMathOperator*{\argmin}{argmin}

\DeclareMathOperator*{\diag}{diag}

\newcommand{\beq}{\begin{equation}}
\newcommand{\eeq}{\end{equation}}
\newcommand{\bq}{\begin{eqnarray}}
\newcommand{\eq}{\end{eqnarray}}
\newcommand{\bqn}{\begin{eqnarray*}}
\newcommand{\eqn}{\end{eqnarray*}}
\newcommand{\bee}{\begin{enumerate}}
\newcommand{\eee}{\end{enumerate}}

\begingroup
\makeatletter
\renewcommand{\p@subfigure}{}
\makeatother

\newlength\fheight
\newlength\fwidth

\begin{document}
\maketitle

\begin{abstract}
In this paper we propose distributed dynamic controllers for sharing both frequency containment and restoration reserves of asynchronous AC systems connected through a multi-terminal HVDC (MTDC) grid. The communication structure of the controller is distributed in the sense that only local and neighboring state information is needed, rather than the complete state. We derive sufficient stability conditions, which guarantee that the AC frequencies converge to the nominal frequency. Simultaneously, a global quadratic power generation cost function is minimized. The proposed controller also regulates the voltages of the MTDC grid, asymptotically minimizing a quadratic cost function of the deviations from the nominal DC voltages. 
The results are valid for distributed cable models of the HVDC grid (e.g. $\pi$-links), as well as AC systems of arbitrary number of synchronous machines, each modeled  by the swing equation. 
We also propose a decentralized, communication-free version of the controller.  
The proposed controllers are tested on a high-order dynamic model of a power system consisting of asynchronous AC grids, modelled as IEEE 14 bus networks, connected through a six-terminal HVDC grid. The performance of the controller is successfully evaluated through simulation.  
\end{abstract}

\section{Introduction}
Power transmission over long distances with low losses is an important challenge as the distances between generation and consumption increase. As the share of fluctuating renewables rises, so does the need to balance generation and consumption mismatches, often over large geographical areas, for which high-voltage direct current (HVDC) power transmission is a commonly used technology. 
In addition to offering lower cost solutions for longer overhead lines and cable transmission \cite{melhem2013electricity}, the controllability of the HVDC converters offers flexibility and means to mitigate problems due to power fluctuations from renewables. 
 Increased use of HVDC technologies for electrical power transmission suggests that future HVDC transmission systems are likely to consist of multiple terminals connected by HVDC transmission lines, so called multi-terminal HVDC (MTDC) systems \cite{Haileselassie2013Power}. 

The fast operation of the DC converters enables frequency regulation of one of the  AC grids connected to the HVDC link. One such example is the frequency regulation of the island of Gotland in Sweden, which is connected to the much stronger Nordic grid through an HVDC cable \cite{arrillaga2007flexible}. 
By connecting multiple AC grids by an MTDC system, enables frequency regulation of one or more of the AC grids connected. Traditional AC frequency controllers and HVDC voltage controllers do however not take advantage which the increased connectivity of the grids brings. Rather than sharing control reserves, each AC area is responsible for maintaining its own frequency in an acceptable range \cite{kundur1994power}, which reduces the need for frequency regulation reserves in the individual AC systems \cite{li2008frequency, Haileselassie2010}. A challenge is to bring back the HVDC grid, e.g. the DC voltages, to a normal operation state after a contingency have happens. 

Stability analysis of combined AC and MTDC systems was performed in \cite{Chaudhuri2011}. 
In \cite{dai2011voltage} and \cite{silva2012provision}, decentralized controllers are employed to share frequency control reserves. In \cite{silva2012provision} no stability analysis is performed, whereas \cite{dai2011voltage} guarantees stability provided that the connected AC areas have identical parameters and the voltage dynamics of the HVDC system are neglected. \cite{taylordecentralized} considers an optimal decentralized controller for AC grids connected by HVDC systems. 

By connecting the AC areas with a communication network supporting the frequency controllers, the performance can be further improved, compared to a decentralized controller structure without such communication. In this paper, we seek to explore controllers which improve performance of existing controllers. For this, we first propose a controller performance measure. 

Several distributed and decentralized controllers for sharing frequency control reserves have been proposed in the literature. 
In \cite{dai2010impact}, a distributed controller, relying on a communication network, was developed to share frequency control reserves of asynchronous AC transmission systems connected through an MTDC system. However, the controller requires a slack bus to control the DC voltage, and is thus only able to share the generation reserves of the non-slack AC areas. 
Another distributed controller is proposed in \cite{dai2013voltage}. Stability is guaranteed, and the need for a slack bus is eliminated. The voltage dynamics of the MTDC system are however neglected. Moreover the implementation of the controller requires every controller to access measurements of the DC voltages of all MTDC terminals. 
In \cite{andreasson2014distributedSecondary, andreasson2015coordinated} distributed secondary generation controllers are proposed, where the MTDC dynamics are explicitly modeled, and the DC voltages are controlled in addition to the frequencies. The controller doeas not rely on a slack bus for controlling the DC voltages. 
The distributed control architecture is more scalable  than a centralized architecture where information from all controllers has to be processed simultaneously. 
By using local and neighboring state information, we propose  controllers, which can be implemented even when communication is unavailable. 
This paper builds on the results in \cite{andreasson2014distributed, andreasson2014distributedSecondary, andreasson2015coordinated}, but significantly generalizes the models of the power system.

 The remainder of this paper is organized as follows. In Section \ref{sec:model}, the system model and the control objectives are defined. In Section \ref{sec:secondary_frequency_control}, a distributed secondary frequency controller for sharing frequency control and restoration reserves is presented, and is shown to satisfy the control objectives. 
 In Sections~\ref{sec:AC_network} and \ref{sec:PI-model}, the results are generalized to general AC networks and $\pi$-link models of the HVDC lines, respectively. 
 In Section \ref{sec:simulations}, simulations of the distributed controller on a six-terminal MTDC test system are provided, showing the effectiveness of the proposed controller. The paper ends with concluding remarks in Section \ref{sec:discussion}.


\section{Model and problem setup}
\label{sec:model}
\subsection{Notation}
\label{subsec:prel}
Let $\mathcal{G}$ be a static, undirected graph. Denote by $\mathcal{V}$ and $\mathcal{E}$ the set of vertices and edges of $\mathcal{G}$, respectively. Let $\mathcal{N}_i$ be the set of neighboring vertices to $i \in \mathcal{V}$.
Denote by $\mathcal{\mathcal{L}}_W$ the weighted Laplacian matrix of $\mathcal{G}$, with edge-weights given by the  elements of the diagonal matrix $W$ \cite{biggs1993algebraic}. 
Let $e_i$ denote the $i$th Cartesian unit vector. 
Let $\mathbb{C}^-$ denote the open left half complex plane, and $\bar{\mathbb{C}}^-$ its closure. We denote by $c_{n\times m}$ an $n\times m$-matrix, whose elements are all equal to $c$. For simplifying notation, we write $c_n$ for $c_{n\times 1}$. 
\subsection{Model and objective}
\label{subsec:model}
We will give here a unified model for an MTDC system interconnected with several asynchronous AC systems.
We consider an MTDC transmission system consisting of $n$ converters, denoted $i=1, \dots, n$, each connected to an AC system, i.e., there are no pure DC nodes of the MTDC grid. The converters are assumed to be connected by an MTDC transmission grid, i.e. there exist only one connected MTDC grid and not several MTDC grids. The node connected to converter $i$ is modelled by
\begin{align}
\begin{aligned}
C_i \dot{V}_i &= -\sum_{j\in \mathcal{N}_i} \frac{1}{R_{ij}}(V_i -V_j) + I_i^{\text{inj}} ,
\end{aligned}
\label{eq:hvdc_coordinated_voltage}
\end{align}
where $V_i$ is the DC voltage of converter node $i$, $C_i>0$ the total capacitance of the converter and the HVDC line connected to the considered converter, and $I_i^{\text{inj}}$  the injected current from the DC converter to the DC node.  The constant $R_{ij}$ denotes the resistance of the HVDC transmission line connecting the converters $i$ and $j$. 
The MTDC transmission grid is assumed to be connected.
Note that the converter model \eqref{eq:hvdc_coordinated_voltage} of the MTDC system does not take the dynamics of the HVDC lines into account, caused by the inductance and capacitance of the lines. In Section~\ref{sec:PI-model}, however, we show that the model \eqref{eq:hvdc_coordinated_voltage} can be generalized to a $\pi$-link model, where each HVDC line consists of an arbitrary number of resistors, inductors, and capacitors in series. 
Only HVDC nodes which are connected to a converter are considered in our model \eqref{eq:hvdc_coordinated_voltage}. This implies that intermediate nodes are not captured by the model. Modelling intermediate nodes would result in differential-algebraic equations, resulting in a far more complex analysis. While systems with intermediate nodes can be transformed into systems without intermediate nodes by Kron reduction \cite{Dorfler2013_Kron}, this is beyond the scope of this paper. 
 Each AC system is assumed to consist of a single generator which is connected to a DC converter, representing an aggregated model of an AC grid. The dynamics of the AC system are given by \cite{machowski2008power}:
\begin{align}
m_i \dot{\omega}_i &=  P^\text{gen}_i + P_i^{{m}} - P_i^{\text{inj}}, \label{eq:frequency}
\end{align}
where $m_i>0$ is its moment of inertia. The constant $P^\text{gen}_i$ is the generated power, $P^m_i$ is the power load, and $P_i^{\text{inj}}$ is the power injected to the DC system through converter $i$, respectively. 
While the model \eqref{eq:frequency} is restricted to single-generator AC systems, we show in Section~\ref{sec:AC_network} that this model can be generalized to a network of arbitrary many generators. \\
The control objective can now be stated as follows.
\begin{objective}
\label{obj:1_hvdc_coordinated}
The frequency deviations are asymptotically equal to zero, i.e.,
\begin{align}
\lim_{t\rightarrow \infty} \omega_i(t)-\omega^{\text{ref}} = 0  \quad  i = 1, \dots, n, \label{eq:hvdc_coordinated_frequency_objective}
\end{align}
where $\omega^{\text{ref}}$ is the nominal frequency. The total quadratic cost of the power generation is minimized asymptotically, i.e., $\lim_{t\rightarrow \infty} P_i^\text{gen} = P_i^{\text{gen}*},  \forall i=1, \dots, n$, where
\begin{align}
[P_1^{\text{gen}*}, \dots, P_n^{\text{gen}*}]  =  \argmin_{P_1, \dots, P_n} \frac 12 \sum_{i=1}^n f^P_i \Big(P_i^\text{gen}\Big)^2 \label{eq:hvdc_coordinated_generation_objective} 
\end{align}
subject to \eqref{eq:hvdc_coordinated_frequency_objective}, i.e., $P^\text{gen}_i + P_i^{{m}} - P_i^{\text{inj}} = 0, \; \forall i=1, \dots, n$ and $\sum_{i=1}^n P_i^\text{inj} = 0$, i.e., power balance both in the AC grids and in the MTDC grid. The positive constants  $f^P_i$ represent the local cost of generating power. 
Finally, the DC voltages are such that the a quadratic cost function of the voltage deviations is minimized asymptotically, i.e., $\lim_{t\rightarrow \infty} V_i = V_i^*, \forall i=1, \dots, n$, where
\begin{align}
 [V_1^*, \dots,  V_n^*] = \argmin_{V_1, \dots, V_n} \frac 12 \sum_{i=1}^n f^V_i (V_i - V_i^\text{ref})^2
 \label{eq:hvdc_coordinated_voltage_objective}
\end{align}
subject to \eqref{eq:hvdc_coordinated_frequency_objective}--\eqref{eq:hvdc_coordinated_generation_objective}, and where the $f^V_i$ is a positive constant reflecting the local cost of DC voltage deviations and $V_i^\text{ref}$ is the nominal DC voltage of converter $i$. 
\end{objective}

\begin{remark}
Note that the order in which the optimization problems \eqref{eq:hvdc_coordinated_generation_objective}--\eqref{eq:hvdc_coordinated_voltage_objective} are solved is crucial, as \eqref{eq:hvdc_coordinated_frequency_objective} and the optimal solution of \eqref{eq:hvdc_coordinated_generation_objective} are constraints of \eqref{eq:hvdc_coordinated_voltage_objective}. 
\end{remark}

\begin{remark}
The minimization of \eqref{eq:hvdc_coordinated_generation_objective} is equivalent to power sharing, where the generated power of AC area $i$ is asymptotically inverse proportional to the cost $f_i^P$. The cost $f_i^P$ can be chosen to reflect the available generation capacity of area $i$.
\end{remark}
\begin{remark}
It is in general not possible that $\lim_{t\rightarrow \infty} V_i(t) = V_i^\text{ref} \; \forall i=1, \dots, n$, since this does not allow for the  currents between the HVDC converters to change by \eqref{eq:hvdc_coordinated_voltage}. Note that the optimal solution to \eqref{eq:hvdc_coordinated_generation_objective} fixes the relative DC voltages, leaving only the ground voltage as a decision variable of \eqref{eq:hvdc_coordinated_voltage_objective}.
Note also that the reference DC voltages $V_i^\text{ref}, \; i=1, \dots, n$, are generally not uniform, as is the reference frequency $\omega^\text{ref}$. 
 \end{remark}
  \begin{remark}
 Note that Objective \ref{obj:1_hvdc_coordinated} does not include constraints of, e.g., generation and line capacities. This requires that the perturbations from the operating point are sufficiently small, to guarantee that these constraints are not violated. 
 \end{remark}

\section{Distributed frequency control}
\label{sec:secondary_frequency_control}

\subsection{Controller structure}
\label{subsec:controller_structure}
In this section we propose a distributed secondary frequency controller. In addition to the generation controller proposed in \cite{andreasson2014distributedSecondary}, we also propose a secondary controller for the power injections into the HVDC grid. We implement the controllers for single AC generators. In Section~\ref{sec:AC_network}, we generalize the controller for AC grids of arbitrary size. 

The distributed generation controller of the AC systems is given by
\begin{align}
P^\text{gen}_i &=- K_i^{\text{droop}} (\omega_i-\omega^{\text{ref}}) -  \frac{K^V_i}{K^\omega_i} K^\text{droop, I}_i \eta_i \nonumber \\
\dot{\eta}_i &= K_i^{\text{droop,I}}(\omega_i-\omega^{\text{ref}}) - \sum_{j\in \mathcal{N}_i} c^\eta_{ij} (\eta_i-\eta_j),
\label{eq:hvdc_coordinated_droop_control_secondary_distributed}
\end{align}
where $K_i^\text{droop}$ and $K_i^\text{droop, I}$ are positive controller parameters. Moreover, $c^\eta_{ij} = c^\eta_{ji}>0$, i.e., the communication graph is supposed to be undirected.
The above controller can be interpreted as a distributed PI-controller, with a distributed averaging filter acting on the integral states $\eta_i$. The first line of Equation~\eqref{eq:hvdc_coordinated_droop_control_secondary_distributed}
 resembles a decentralized droop controller with a setpoint given by $\eta_i$. The second line of Equation~\eqref{eq:hvdc_coordinated_droop_control_secondary_distributed} updates the variable $\eta_i$ in a distributed fashion by a distributed averaging integral controller. 
 The magnitudes of the variables $c_{ij}^\eta$ determine how fast the generated power levels converge. While a larger magnitude of $c_{ij}^\eta$ could lead to faster convergence of the generated power, it can also induce oscillations. 
 It is possible to implement a decentralized version of \eqref{eq:hvdc_coordinated_droop_control_secondary_distributed} by dropping the states $\eta_i$. This results in the following controller
\begin{align}
P^\text{gen}_i &=- K_i^{\text{droop}} (\omega_i-\omega^{\text{ref}}). 
\label{eq:hvdc_coordinated_droop_control_secondary_distributed_decentralized_version}
\end{align}
The proposed converter controllers governing the power injections from the AC systems into the HVDC grid are given by
\begin{align}
\label{eq:voltage_control_secondary}
P_i^{\text{inj}} &=  K_i^{{\omega}} (\omega_i - \omega^{\text{ref}}) + K_i^{{V}}(V_i^{\text{ref}}-V_i) \nonumber \\ 
& \;\;\;\;+ \sum_{j\in \mathcal{N}_i} c^\phi_{ij} (\phi_i - \phi_j) \nonumber \\
\dot{\phi}_i &= \frac{K^\omega_i}{K^V_i} \omega_i - \gamma \phi_i, 
\end{align}
where  $K_i^V$ and $K_i^\omega$ are positive controller parameters, and $P_i^{\text{inj, nom}}$ is the nominal injected power, $\gamma\ge 0$ and $c^\phi_{ij} = c^\phi_{ji}>0$. 
If $\gamma=0$, the converter controller \eqref{eq:voltage_control_secondary} can be interpreted as an emulation of an AC network between the isolated AC areas, as it resembles the swing equation. The auxiliary controller variables $\phi_i$ are then equivalent to the phase angles of AC area $i$, whose differences govern the power transfer between the areas. 
Larger magnitudes of $c^\phi_{ij}$ correspond to higher conductances of the AC lines, and thus stronger coupling and faster synchronization of the frequencies. 
If $\gamma > 0$, damping is added to the dynamics of $\phi_i$. Damping generally improves stability margins, and turns out to be very useful in the stability analysis. However, a nonzero $\gamma$ also implies that the AC dynamics are not emulated perfectly. This implies that exact frequency synchronization might not be possible in general. 
 In contrast to a connection with AC lines, the power is fed into the MTDC grid and then transfered to the other AC areas through the MTDC grid rather than through an AC grid. Also the converter controller can be implemented in a decentralized version by dropping the states $\phi_i$, resulting in the following controller
\begin{align}
\label{eq:voltage_control_secondary_decentralized_version}
P_i^{\text{inj}} &=  K_i^{{\omega}} (\omega_i - \omega^{\text{ref}}) + K_i^{{V}}(V_i^{\text{ref}}-V_i). 
\end{align}
The HVDC converter response is assumed to be instantaneous, i.e., injected power on the AC side is immediately and losslessly converted to DC power.  This assumption is reasonable due to the dynamics of the converter typically being orders of magnitudes faster than the primary frequency control dynamics of the AC system \cite{kundur1994power}.
The relation between the injected HVDC current and the injected AC power is thus given by
\begin{align}
V_iI_i^{\text{inj}} = P_i^{\text{inj}}. \label{eq:power-current_nonlinear}
\end{align}
By assuming $V_i=V^{\text{nom}}\;  i=1, \dots, n$, where $V^{\text{nom}}$ is a global nominal DC voltage, we obtain
\begin{align}
V^{\text{nom}}I_i^{\text{inj}} = P_i^{\text{inj}}. \label{eq:power-current}
\end{align}
Assumption \eqref{eq:power-current} relies on the assumption that the voltages $V_i$ do not deviate significantly from the nominal voltage $V^{\text{nom}}$. Since for most HVDC converters the acceptable deviation from the nominal voltage is less than $5 \%$ \cite{jovcic2015high}, the approximation \eqref{eq:power-current} would result in a relative error smaller than $5 \%$.

\subsection{Stability analysis}
\label{sec:secondary_frequency_control_stability}
We now analyze the stability of the closed-loop system. Define the state vectors $\hat{\omega}=\omega - \omega^\text{ref}1_n$ and $\hat{V}=V - V^\text{ref}$, where $\omega=[\omega_1, \dots, \omega
_n]^T$, $V=[V_1, \dots, V_n]^T$, $V^\text{ref} = [V_1^\text{ref}, \dots, V_n^\text{ref}]^T$, $\eta = [\eta_1, \dots, \eta_n]^T$, and $\phi = [\phi_1, \dots, \phi_n]$. 
Combining the MTDC \eqref{eq:hvdc_coordinated_voltage}, the AC dynamics \eqref{eq:frequency} with the generation control \eqref{eq:hvdc_coordinated_droop_control_secondary_distributed}, the converter controller \eqref{eq:voltage_control_secondary} with the power-current relationship \eqref{eq:power-current}, we obtain the closed-loop dynamics
\begin{IEEEeqnarray}{lcl}
\dot{ \hat{\omega} } &=& M \Big(- (K^{\text{droop}} + K^\omega) \hat{\omega} + K^V \hat{V} \nonumber \\
&& - {K^V}(K^\omega)^{-1} K^\text{droop, I} \eta - \mathcal{L}_\phi \phi  + P^{{m}}  \Big) \nonumber \\
\dot{\hat{V}} &=& 
\frac{1}{V^{\text{nom}}}E {K}^\omega \hat{\omega} -E\left(\mathcal{L}_R + \frac{K^V}{V^{\text{nom}}} \right) \hat{V} + \frac{1}{V^\text{nom}} E \mathcal{L}_\phi  \phi \nonumber \\
\dot{\eta} &=&  K^{\text{droop,I}}\hat{\omega} - \mathcal{L}_\eta \eta \nonumber \\
\dot{\phi} &=&  (K^V)^{-1}{K}^\omega \hat{\omega}  -\gamma \phi, \label{eq:cl_dynamics_vec_delta_int_coordinated}
\end{IEEEeqnarray}
where 
$M=\diag({m_1}^{-1}, \hdots , {m_n}^{-1})$ is a matrix of inverse generator inertia, 
 $E=\diag(C_1^{-1}, \dots, C_n^{-1})$ is a matrix of electrical elastances,  $\mathcal{L}_R$ is the weighted Laplacian matrix  of the MTDC grid with edge-weights $1/R_{ij}$,  $\mathcal{L}_\eta$ and $\mathcal{L}_\phi$ are the weighted Laplacian matrices of the communication graphs with edge-weights $c^\eta_{ij}$ and $c^\phi_{ij}$, respectively, and $P^m=[P^m_1, \dots, P^m_n]^T$. We define the diagonal matrices of the controller gains as $K^\omega = \diag(K^\omega_1, \dots, K^\omega_n)$, etc. 

Let $y=[\hat{\omega}^T, \hat{V}^T]^T$ define the output of \eqref{eq:cl_dynamics_vec_delta_int_coordinated}. Clearly the linear combination $1_n^T\phi$ is unobservable and marginally stable with respect to the dynamics \eqref{eq:cl_dynamics_vec_delta_int_coordinated}, as it lies in the nullspace of $\mathcal{L}_\phi$. In order to facilitate the stability analysis, we will perform a state-transformation to this unobservable mode. Consider the following state-transformation:
\begin{align}
\phi' = \begin{bmatrix}
\frac{1}{\sqrt{n}} 1_n^T \\ S^T
\end{bmatrix} \phi
\qquad
\phi = \begin{bmatrix}
\frac{1}{\sqrt{n}} 1_n & S
\end{bmatrix} \phi'
\label{eq:transformation_phi}
\end{align}
where $S$ is an $n\times(n-1)$ matrix such that $\left[\frac{1}{\sqrt{n}}1_n \; S\right]$ is orthonormal. By applying the state-transformation \eqref{eq:transformation_phi} to \eqref{eq:cl_dynamics_vec_delta_int_coordinated}, we obtain dynamics where it can be shown that 
the state $\phi_1'$ is unobservable with respect to the defined output. Hence, omitting $\phi_1'$ does not affect the output dynamics. Thus, we define $\phi''=[\phi_2', \dots, \phi_n']$, and obtain the dynamics
\begin{IEEEeqnarray}{lcl}
\dot{\hat{\omega}} &=& M \Big(- (K^{\text{droop}} + K^\omega) \hat{\omega} + K^V \hat{V} \nonumber \\
&& - {K^V}(K^\omega)^{-1} K^\text{droop, I} \eta - \mathcal{L}_\phi S \phi''  + P^{{m}}  \Big) \nonumber \\
\dot{\hat{V}} &=& 
\frac{1}{V^{\text{nom}}}E {K}^\omega \hat{\omega} -E\left(\mathcal{L}_R + \frac{K^V}{V^{\text{nom}}} \right) \hat{V} + \frac{1}{V^\text{nom}} E \mathcal{L}_\phi  S\phi''  \nonumber \\
\dot{\eta} &=&  K^{\text{droop,I}}\hat{\omega} - \mathcal{L}_\eta \eta \nonumber \\
\dot{\phi}'' &=&  S^T(K^V)^{-1}{K}^\omega \hat{\omega}  -\gamma \phi''. \label{eq:cl_dynamics_vec_delta_int_coordinated_double_prime}
\end{IEEEeqnarray}
We are now ready to show the main stability result of this section. The following assumptions are later used as sufficient conditions for closed-loop stability. 
\begin{assumption}
\label{ass:L_phi_coordinated}
The Laplacian matrix satisfies $\mathcal{L}_\phi = k_\phi \mathcal{L}_R$.
\end{assumption}
Assumption \ref{ass:L_phi_coordinated} can be interpreted as the emulated AC dynamics of \eqref{eq:voltage_control_secondary} having the same susceptance ratios as the conductance ratios of the HVDC lines. 
Assumption~\ref{ass:L_phi_coordinated}  can always be satisfied by appropriate choices of
the constants $c_{ij}$ in \eqref{eq:voltage_control_secondary}.
\begin{assumption}
\label{ass:gamma_coordinated}
The gain $\gamma$ satisfies $\gamma > {k_\phi}/({4V^\text{nom}})$. 
\end{assumption}
Assumption \ref{ass:gamma_coordinated} lower bounds for the damping coefficient of the converter controllers. Note that the bound on $\gamma$ is independent of the topology of the communication network. This is particularly desirable in a plug-and-play setting, where new nodes can be added to the system, without having to change $\gamma$. 
\begin{theorem}
\label{th:stability_A_coordinated}
If Assumptions \ref{ass:L_phi_coordinated} and \ref{ass:gamma_coordinated} hold, the equilibrium of \eqref{eq:cl_dynamics_vec_delta_int_coordinated_double_prime} is globally asymptotically stable. 
\end{theorem}
\begin{proof}
The proof follows from  Theorem~\ref{th:hvdc_coordinated_voltage_control_secondary_network_stability}, and is thus omitted. 
\end{proof}
\begin{corollary}
\label{cor:hvdc_coordinated_equilibrium}
Let Assumption~\ref{ass:L_phi_coordinated} hold and let $\gamma$, $k_\phi$ be given such that Assumption~\ref{ass:gamma_coordinated} holds. 
Let $K^V, K^\omega$ and $K^\text{droop, I}$ be such that $(F^P)^{-1} = K^V(K^\omega)^{-1}K^\text{droop, I}$ and $F^V=K^V$, where $F^P = \diag(f^P_1, \dots, f^P_n)$ and $F^V = \diag(f^V_1, \dots, f^V_n)$. 
 Then the dynamics \eqref{eq:cl_dynamics_vec_delta_int_coordinated_double_prime} satisfy Objective~\ref{obj:1_hvdc_coordinated} in the limit when $\norm{(K^\omega)^{-1} K^V}_\infty \rightarrow 0$, provided that the disturbance $P^m_i$ is constant. 
\end{corollary}
\begin{proof}
By Theorem~\ref{th:stability_A_coordinated},  \eqref{eq:cl_dynamics_vec_delta_int_coordinated_double_prime} has a unique and  stable equilibrium. 
Letting $\dot{\phi}''=0_{n-1}$ implies
$S^T(K^V)^{-1}K^\omega \hat{\omega} - \gamma \phi'' = 0_{n-1} 
$. Now $\norm{(K^\omega)^{-1} K^V}_\infty \rightarrow 0$ implies that $S^T\hat{\omega} = 0 \Leftrightarrow \hat{\omega} = k_1 1_n$ for some $k_1\in \mathbb{R}$. Letting $\dot{\eta}=0_n$ in  \eqref{eq:cl_dynamics_vec_delta_int_coordinated_double_prime} yields
\begin{align*}
K^\text{droop, I} \hat{\omega} - \mathcal{L}_\eta \eta = 0_n.
\end{align*}
By inserting $\hat{\omega} = k_1 1_n$ and premultiplying the above equation with $1_n^T$, we obtain that $k_1=0$, so $\hat{\omega}=0_n$ so  Equation~\eqref{eq:hvdc_coordinated_frequency_objective} of Objective~\ref{obj:1_hvdc_coordinated} is thus satisfied. Thus   $\eta = k_2 1_n$ for some $k_2\in \mathbb{R}$.  Finally, we let $\dot{\hat{V}}=0_n$ in \eqref{eq:cl_dynamics_vec_delta_int_coordinated_double_prime}:
\begin{align}
K^\omega \hat{\omega} - \Big( V^\text{nom} \mathcal{L}_R + K^V \Big) \hat{V} +  \mathcal{L}_\phi S \phi'' = 0_n. \label{eq:cl_dynamics_vec_delta_int_coordinated_double_prime_n+1:2n}
\end{align}
Inserting $\hat{\omega}=0_n$ and premultiplying \eqref{eq:cl_dynamics_vec_delta_int_coordinated_double_prime_n+1:2n} with $1_n^T$ yield  
\begin{align}
1_n^TK^V\hat{V}=0.
\label{eq:hvdc_coordinated_voltage_equilibrium}
\end{align}
Inserting $\hat{\omega}=0_n$ and $\eta = k_2 1_n$ in \eqref{eq:hvdc_coordinated_droop_control_secondary_distributed} yields 
\begin{align}
P^\text{gen} = - k_2 K^V(K^\omega)^{-1} K^\text{droop, I} 1_n,
\label{eq:hvdc_coordinated_eta_equilibrium}
\end{align}
where $P^\text{gen} = [P^\text{gen}_1, \dots, P^\text{gen}_n]^T$.
It now remains to show that the equilibrium of \eqref{eq:cl_dynamics_vec_delta_int_coordinated_double_prime} minimizes the  cost functions \eqref{eq:hvdc_coordinated_generation_objective}  and \eqref{eq:hvdc_coordinated_voltage_objective} of Objective~\ref{obj:1_hvdc_coordinated}. Consider first \eqref{eq:hvdc_coordinated_generation_objective}, with the constraints
$P^\text{gen}_i  + P_i^{{m}} - P_i^{\text{inj}} = 0, i=1, \dots, n$ and $\sum_{i=1}^n P_i^\text{inj} = 0$. By summing the first constraints we obtain $\sum_{i=1}^n P^\text{gen}_i  = - \sum_{i=1}^n  P_i^{{m}} $. The KKT condition of \eqref{eq:hvdc_coordinated_generation_objective} is 
\begin{align}
F^P P^\text{gen} = - k_3 1_n.
\label{eq:hvdc_coordinated_KKT_generation}
\end{align}
Since $(F^P)^{-1} = K^V(K^\omega)^{-1}K^\text{droop, I}$,  \eqref{eq:hvdc_coordinated_eta_equilibrium} and \eqref{eq:hvdc_coordinated_KKT_generation} are identical for $k_2=k_3$. We conclude that \eqref{eq:hvdc_coordinated_generation_objective} is minimized. 
Since $P^\text{gen} =  - K^V(K^\omega)^{-1} K^\text{droop, I} \eta = - k_2 K^V(K^\omega)^{-1} K^\text{droop, I} 1_n$ and $\hat{\omega}=0_n$, premultiplying the first $n$ rows of the equilibrium of \eqref{eq:cl_dynamics_vec_delta_int_coordinated_double_prime} with $M^{-1}$, and adding to the $(n+1)$th to $2n$th rows premultiplied with $V^\text{nom}E^{-1}$ yields
\begin{align*}
- V^\text{nom} \mathcal{L}_R \hat{V} - k_2 K^V(K^\omega)^{-1} K^\text{droop, I} 1_n = P^m.
\end{align*}
Premultiplying the above equation with $1^T_n$ yields $k_2 = - \sum_{i=1}^n P^m_i \sum_{i=1}^n \frac{K^\omega_i}{K^V_i K^\text{droop}_i}$. Additionally, $\mathcal{L}_R \hat{V}$ is uniquely determined. 
Now consider \eqref{eq:hvdc_coordinated_voltage_objective}. Note that $P_i^\text{inj}$ and hence $I_i^\text{inj}$, are uniquely determined by \eqref{eq:hvdc_coordinated_generation_objective}. By the equilibrium of \eqref{eq:hvdc_coordinated_voltage}, $\mathcal{L}_R \hat{V} = I^\text{inj}$, where $I^\text{inj} = [I^\text{inj}_1, \dots, I^\text{inj}_n]^T$. Thus, the KKT condition of \eqref{eq:hvdc_coordinated_voltage_objective} is 
\begin{align}
F^V \hat{V} =  \mathcal{L}_R r,
\label{eq:hvdc_coordinated_KKT_voltage_raw}
\end{align}
where $r\in \mathbb{R}^n$. 
Since $\mathcal{L}_R \hat{V}$ is uniquely determined, we premultiply \eqref{eq:hvdc_coordinated_KKT_voltage_raw} with $1_n^T$ and obtain the equivalent condition
\begin{align}
1^T_n F^V \hat{V} = 0.
\label{eq:hvdc_coordinated_KKT_voltage}
\end{align}
Since $F^V = K^V$, \eqref{eq:hvdc_coordinated_voltage_equilibrium} and \eqref{eq:hvdc_coordinated_KKT_voltage} are equivalent. Hence \eqref{eq:hvdc_coordinated_voltage_objective} is minimized, so Objective~\ref{obj:1_hvdc_coordinated} is satisfied. 
\end{proof} 
\begin{remark}
Corollary \ref{cor:hvdc_coordinated_equilibrium} provides insight in choosing the controller gains of \eqref{eq:hvdc_coordinated_droop_control_secondary_distributed} and \eqref{eq:voltage_control_secondary}, to satisfy Objective~\ref{obj:1_hvdc_coordinated}. 
\end{remark}

While the generation controller \eqref{eq:hvdc_coordinated_droop_control_secondary_distributed} and the converter controller \eqref{eq:voltage_control_secondary} offer good performance in terms of satisfying Objective~\ref{obj:1_hvdc_coordinated}, it may not be possible to implement these distributed controllers, e.g., due to lack of communication infrastructure. For such MTDC systems where appropriate communication is lacking, it may be desirable to instead implement decentralized generation and converter controllers. In other situations it might be possible to implement the distributed generation controller \eqref{eq:hvdc_coordinated_droop_control_secondary_distributed}, while it is more desirable to have the HVDC converters operating independently with decentralized controllers.
 In the following corollary, we show that the decentralized generation and converter controllers \eqref{eq:hvdc_coordinated_droop_control_secondary_distributed_decentralized_version} and \eqref{eq:voltage_control_secondary_decentralized_version} also globally asymptotically stabilize the combined MTDC and AC system. 
\begin{corollary}
\label{cor:hvdc_coordinated_stability_decentralized}
Let Assumption~\ref{ass:L_phi_coordinated} hold and let $\gamma$, $k_\phi$ be given such that Assumption~\ref{ass:gamma_coordinated} holds. 
Consider the dynamics of the MTDC dynamics \eqref{eq:hvdc_coordinated_voltage} and the AC dynamics \eqref{eq:frequency} with the generation controller \eqref{eq:hvdc_coordinated_droop_control_secondary_distributed_decentralized_version}
or \eqref{eq:hvdc_coordinated_droop_control_secondary_distributed}, respectively, 
and the converter controller
\eqref{eq:voltage_control_secondary_decentralized_version}. 
 The equilibria of the resulting closed-loop systems are globally asymptotically stable. 
\end{corollary}
\begin{proof}
The proof is in line with the proof of Theorem~\ref{th:hvdc_coordinated_voltage_control_secondary_network_stability}, where we discard the variables $\eta$ and $\phi$. 
\end{proof}
While the optimality results of Corollary~\ref{cor:hvdc_coordinated_equilibrium} do not hold for any other controller combinations than \eqref{eq:hvdc_coordinated_droop_control_secondary_distributed} and \eqref{eq:voltage_control_secondary}, the following remark can be made about the average frequency errors. 
\begin{lemma}
Consider the dynamics of the MTDC dynamics \eqref{eq:hvdc_coordinated_voltage} and the AC dynamics \eqref{eq:frequency} with the generation controller \eqref{eq:hvdc_coordinated_droop_control_secondary_distributed} and the converter controller \eqref{eq:voltage_control_secondary}. Any equilibrium of the resulting closed-loop system satisfies $\sum_{i=1}^n K^{\text{droop,I}}_i (\omega_i - \omega^\text{ref}) = 0$, i.e., the average frequency errors are zero. 
\end{lemma}
\begin{proof}
Consider the closed-loop dynamics \eqref{eq:cl_dynamics_vec_delta_int_coordinated}. Letting $\dot{\eta}=0_n$ and premultiplying this equation with $1_n^T$ yields
$0_n = 1_n^T K^{\text{droop,I}}\hat{\omega} - 1_n^T \mathcal{L}_\eta \eta = \sum_{i=1}^n K^{\text{droop,I}}_i (\omega_i - \omega^\text{ref})$.
\end{proof}
\section{Generalisation to AC generation network}
\label{sec:AC_network}
In this section we generalize the single-generator model of Section \ref{subsec:model} to an AC grid with arbitrary size.
\subsection{Objective}
 Consider the AC transmission grid connected to converter $i$, and suppose it consists of $n_i$ generator buses. Without loss of generality, we may assume that converter $i$ of the MTDC grid is connected to generator $i_1$ of the AC system $i$. Let $\delta_{i_k}$ be the phase angle of bus $i_k$. The dynamics of the power system are assumed to be given by the linearized swing equation \cite{machowski2008power}, where the voltages are assumed to be constant. As before, we consider the incremental states with respect to their reference values:
\begin{align}
 \dot{\delta}_{i_k} &= \hat{\omega}_{i_k} \nonumber \\
m_{i_k}\dot{\hat{\omega}}_{i_k} &= - (K^{\text{droop}}_{i_k} + K^\omega_{i_k}) \hat{\omega}_{i_k}  -\sum_{j\in \mathcal{N}_{i_k}} k_{{i_k}j}(\delta_{i_k} - \delta_j) \nonumber \\
& \;\;\;\; + P^\text{gen}_{i_k} + P_{i_k}^{{m}} - P_{i_k}^{\text{inj}},
\label{eq:hvdc_coordinated_swing_scalar}
\end{align}
where ${\delta}_{i_k}$ is the phase angle and $\hat{\omega}_{i_k}={\omega}_{i_k}-\omega^{\text{ref}}$ is the incremental frequency at bus $i_k$,  $m_{i_k}>0$ is the inertia of bus $i_k$, $k_{i_kj} = |V_{i_k}||V_j|b_{i_kj}$, where $V_i$ is the constant voltage of bus $i$, and $b_{i_kj}$ is the susceptance of the power line $(i_k,j)$. Moreover $K^{\text{droop}}_{i_k} =0$ for $k\ne 1$, since power injection through the HVDC converter only takes place at bus  $i_1$. 
The constant $P^\text{gen}_{i_k}$ is the generated power by the generation control, $P^m_{i_k}$ is the uncontrolled deviation from the nominal generated power at generator ${i_k}$, respectively. The variable $P_{i_k}^{\text{inj}}=0$ for $k\ne 1$ is the power injected to the DC system through converter ${i}$. 
We assume that the AC voltages are constant, thus implying that $k_{ij}$ is constant.  
In order to account for the additional generators, we need to slightly modify Objective~\ref{obj:1_hvdc_coordinated}.
\begin{objective}
\label{obj:1_hvdc_coordinated_ac_network}
The frequency deviations converge to zero, i.e.,
\begin{align}
\lim_{t\rightarrow \infty} \omega_{i_k}(t)-\omega^{\text{ref}} = 0  \quad k=1, \dots n_i, \; i = 1, \dots, n. \label{eq:hvdc_coordinated_frequency_objective_ac_network}
\end{align}
The total cost of the power generation is minimized asymptotically, i.e., $\lim_{t\rightarrow \infty} P_i^\text{gen} = P_i^{\text{gen}*}, i=1, \dots, n$, where
\begin{align}
[P_1^{\text{gen}*}, \dots, P_n^{\text{gen}*}]  =  \argmin_{P_1, \dots, P_n} \frac 12 \sum_{i=1}^n P_i^T f^P_i P_i \label{eq:hvdc_coordinated_generation_objective_ac_network} 
\end{align}
subject to $1^T_{n_i}(P^\text{gen}_i + P_i^{{m}} - P_i^{\text{inj}}) = 0, i=1, \dots, n$ and $\sum_{i=1}^n P_{i_1}^\text{inj} = 0$, i.e., power balance both in the AC grids and in the MTDC grid. Here $P^\text{gen}_{i} = [P^\text{gen}_{i_1}, \dots, P^\text{gen}_{i_{n_i}}]^T, \; i=1, \dots, n$. 
Finally, the DC voltages are such that a quadratic cost function of the voltage deviations is minimized asymptotically, i.e., $\lim_{t\rightarrow \infty} V_i = V_i^*,  i=1, \dots, n$, where
\begin{align}
 [V_1^*, \dots,  V_n^*] = \argmin_{V_1, \dots, V_n} \frac 12 \sum_{i=1}^n f^V_i (V_i - V_i^\text{ref})^2
 \label{eq:hvdc_coordinated_voltage_objective_ac_network}
\end{align}
subject to \eqref{eq:hvdc_coordinated_frequency_objective}--\eqref{eq:hvdc_coordinated_generation_objective}. 
Here $f^P_i$ and $f^V_i$ are positive constants. 
\end{objective}
\subsection{Controller structure}
\label{subsec:hvdc_coordinated_controller_structure}
In this section we generalize the distributed secondary frequency controller \eqref{eq:hvdc_coordinated_droop_control_secondary_distributed} and the converter controller \eqref{eq:voltage_control_secondary} to the full AC network. 
The distributed generation controllers of the AC network $i$ are in this case given by
\begin{align}
P^\text{gen}_{i_k} &=- K_{i_k}^{\text{droop}} \hat{\omega}_{i_k} -  \frac{K^V_{i}}{K^\omega_{i_1}} K^\text{droop, I}_{i_k} \eta_i, \; k=1, \dots, n_i \nonumber \\
\dot{\eta}_i &= \sum_{k=1}^{n_i} K_{i_k}^{\text{droop,I}}\hat{\omega}_{i_k} - \sum_{j\in \mathcal{N}_i} c^\eta_{ij} (\eta_i-\eta_j).
\label{eq:hvdc_coordinated_droop_control_secondary_distributed_network}
\end{align}
where $K_{i_k}^\text{droop}$, $K_i^V$, $K_{i_1}^\omega$ and $K_i^\text{droop, I}$ are positive controller parameters, and $c^\eta_{ij} = c^\eta_{ji}>0$. Compare Equation~\eqref{eq:hvdc_coordinated_droop_control_secondary_distributed}.
The above controller can be interpreted as a distributed PI-controller, with a distributed consensus filter acting on the integral states $\eta_i$. 
The converter controller governing the power injections from bus $i_1$ of the AC system $i$ into the HVDC grid is given by
\begin{align}
\label{eq:hvdc_coordinated_voltage_control_secondary_network}
P_i^{\text{inj}} &= P_i^{\text{inj, nom}} + K_{i_1}^{{\omega}} (\omega_{i_1} - \omega^{\text{ref}}) + K_i^{{V}}(V_i^{\text{ref}}-V_i) \nonumber \\ 
& \;\;\;\;+ \sum_{j\in \mathcal{N}_i} c^\phi_{ij} (\phi_i - \phi_j) \nonumber \\
\dot{\phi}_i &= \frac{K^\omega_{i_1}}{K^V_i} (\omega_{i_1}-\omega^\text{ref}) - \gamma \phi_i, 
\end{align}
where $\gamma>0$ and $c^\phi_{ij} = c^\phi_{ji}>0$. Compare Equation~\eqref{eq:voltage_control_secondary}. 
In vector-form \eqref{eq:hvdc_coordinated_swing_scalar} becomes
\begin{align}
\dot{\delta}_i &=  \hat{\omega}_{i} \nonumber \\
\dot{\hat{\omega}}_{i} &= M_i \big(- (K^{\text{droop}}_{i} + K^\omega_{i}) \hat{\omega}_{i} -\mathcal{L}^\text{AC}_{i}\delta_i \nonumber \\
& \;\;\;\; + P^\text{gen}_{i} + P_{i}^{{m}} - P_{i}^{\text{inj}} \big), 
\label{eq:hvdc_coordinated_swing_vector}
\end{align}
where $\delta_i=[\delta_{i_1}, \dots, \delta_{i_{n_i}}]^T$, $\omega_i =[\omega_{i_1}, \dots, \omega_{i_{n_i}}]^T$, $M_i=\diag(m^{-1}_{i_1}, \dots, m^{-1}_{i_{n_i}})$, $\mathcal{L}^\text{AC}_i$ is the Laplacian matrix of the graph corresponding to the AC transmission system, with edge-weights given by $k_{i_kj}$, $K_i^\text{droop} = \diag(K_{i_1}^\text{droop}, \dots, K_{i_{n_i}}^\text{droop})$, $K^\omega_i=\diag(K^\omega_{i_1}, 0, \dots, 0)$, and $P^\text{gen}_{i} = [P^\text{gen}_{i_1}, \dots, P^\text{gen}_{i_{n_i}}]^T$, etc. Consider the output $y_i=\hat{\omega}_i$ of \eqref{eq:hvdc_coordinated_swing_vector}. With respect to $y_i=\hat{\omega}_i$, the dynamics have a marginally stable  unobservable mode. Thus, similar to Section~\ref{sec:secondary_frequency_control_stability} we consider the state transformation
\begin{align*}
\delta_i = \begin{bmatrix}
\frac{1}{\sqrt{i_n}} 1_{n} & S_i
\end{bmatrix}
\delta_i' \qquad
\delta'_i =
\begin{bmatrix}
\frac{1}{\sqrt{i_n}} 1_{n}^T  \\
S_i^T
\end{bmatrix}
\delta_i,
\end{align*}
where $S_i$ is an $i_n\times (i_n-1)$-matrix such that $[\frac{1}{\sqrt{i_n}} 1_{n}, S_i]$ is orthonormal. 
It can be shown that $\delta_{i_1}'$ is unobservable, and can be omitted by introducing the state $\delta_i'' = [\delta_{i_2}', \dots, \delta_{i_{n_i}}']^T$. This state-transformation results in the dynamics
\begin{align}
\dot{\delta}''_i &= 
S_i^T
\hat{\omega}_{i} \nonumber \\
\dot{\hat{\omega}}_{i} &= M_i \Big(- (K^{\text{droop}}_{i} + K^\omega_{i}) \hat{\omega}_{i} -\mathcal{L}^\text{AC}_{i} S_i
\delta_i''  \nonumber \\
& \;\;\;\; + P^\text{gen}_{i} + P_{i}^{{m}} - P_{i}^{\text{inj}} \Big).
\label{eq:hvdc_coordinated_swing_vector''}  
\end{align}
Since the input-output dynamics of \eqref{eq:hvdc_coordinated_swing_vector} and \eqref{eq:hvdc_coordinated_swing_vector''} are identical, we henceforth only consider the dynamics \eqref{eq:hvdc_coordinated_swing_vector''}. 
By combining the dynamics \eqref{eq:hvdc_coordinated_voltage}  and \eqref{eq:hvdc_coordinated_swing_vector''} with the controllers \eqref{eq:hvdc_coordinated_droop_control_secondary_distributed_network} and \eqref{eq:hvdc_coordinated_voltage_control_secondary_network}, and considering the change of coordinates \eqref{eq:transformation_phi} and $\phi''=[\phi_2', \dots, \phi_n']$ we obtain the dynamics
\begin{IEEEeqnarray}{lcl}
\dot{\delta}''_i &=& 
S_i^T
\hat{\omega}_{i}, \; i=1, \dots, n \nonumber \\
\dot{\hat{\omega}}_{i} &=& M_i \Big(- (K^{\text{droop}}_{i} + K^\omega_{i}) \hat{\omega}_{i} + e_1 K_i^V \hat{V}_i -\mathcal{L}^\text{AC}_{i} S_i
\delta_i''  \nonumber \\
&& - \frac{K^V_{i}}{K^\omega_{i_1}} K^\text{droop, I}_{i} 1_{n_i} \eta_i - e_1 e_i^T \mathcal{L}_\phi S \phi''  + P_{i}^{{m}}  \Big), \; i=1, \dots, n \nonumber \\
\dot{\hat{V}} &=& 
\frac{1}{V^{\text{nom}}}E\tilde{K}^\omega \tilde{\omega} -E\left(\mathcal{L}_R + \frac{K^V}{V^{\text{nom}}} \right) \hat{V} + \frac{1}{V^\text{nom}} E \mathcal{L}_\phi S \phi'' \nonumber \\
\dot{\eta} &=& \sum_{i=1}^n e_i 1_n^T K_{i}^{\text{droop,I}}\hat{\omega}_{i} - \mathcal{L}_\eta \eta \nonumber \\
\dot{\phi}'' &=&  S^T(K^V)^{-1}\tilde{K}^\omega \tilde{\omega}  -\gamma I_{n-1} \phi'',
\label{eq:hvdc_coordinated_voltage_control_secondary_network_cl_dynamics}
\end{IEEEeqnarray}
where $\tilde{\omega}=[\hat{\omega}_{1_1}, \dots, \hat{\omega}_{n_1}]^T$, $K_{i}^{\text{droop}} = \diag( K_{i_1}^{\text{droop}}, \dots,  K_{i_{n_i}}^{\text{droop}})$, $ K_{i}^{\text{droop, I}} = \diag( K_{i_1}^{\text{droop, I}}, \dots,  K_{i_{n_i}}^{\text{droop, I}})$, $\tilde{K}^\omega=\diag(K^\omega_{1_1}, \dots, K^\omega_{n_1})$. 
\begin{theorem}
The equilibrium of the dynamics \eqref{eq:hvdc_coordinated_voltage_control_secondary_network_cl_dynamics} is globally asymptotically stable under Assumptions \ref{ass:L_phi_coordinated} and \ref{ass:gamma_coordinated}. 
\label{th:hvdc_coordinated_voltage_control_secondary_network_stability}
\end{theorem}
\begin{corollary}
\label{cor:hvdc_coordinated_equilibrium_ac_network}
Let Assumption~\ref{ass:L_phi_coordinated} hold and let $\gamma$, $k_\phi$ be given such that Assumption~\ref{ass:gamma_coordinated} holds. 
Let $K^V_i, K^\omega_i$ and $K^\text{droop}_i$ be such that $(F^P_i)^{-1} = K^V_i(K^\omega_{i_1})^{-1}K^\text{droop}_i, \; i=1, \dots, n$ and $F^V=K^V$, where $F^P = \diag(f^P_1, \dots, f^P_n)$ and $F^V = \diag(f^V_1, \dots, f^V_n)$. 
 Then Objective~\ref{obj:1_hvdc_coordinated_ac_network} is satisfied in the limit when $\norm{(K^\omega)^{-1} K^V}_\infty \rightarrow 0$, provided that the disturbance $P^m_i$ is constant. 
\end{corollary}
\begin{proof}
Consider  \eqref{eq:hvdc_coordinated_voltage_control_secondary_network_cl_dynamics}. Letting $\delta_i''=0_{n_i}\;i=1,\dots,n$, yields $\hat{\omega}_i = k_i 1_{n_i}\; i=1,\dots,n$. 
 Letting $\dot{\phi}''=0_n$ yields
\begin{align*}
S^T(K^V)^{-1}\tilde{K}^\omega \tilde{\omega}  -\gamma I_{n-1} = 0_n . 
\end{align*}
Now $\norm{(K^\omega)^{-1} K^V}_\infty \rightarrow 0$ in the above equation implies $S^T\tilde{\omega} = 0 \Leftrightarrow \tilde{\omega} = k 1_n, \; k\in \mathbb{R}$. This implies $\hat{\omega}_i = k 1_{n_i},  i=1,\dots,n$. Letting $\dot{\eta}=0_n$, inserting $\hat{\omega}_i = k 1_{n_i}\; i=1,\dots,n$, and premultiplying the equation with $1_n^T$ finally yields $k=0$, and thus $\hat{\omega}_i = 0_{n_i},  i=1,\dots,n$, i.e., \eqref{eq:hvdc_coordinated_frequency_objective_ac_network} is satisfied. 
Letting $\dot{\eta}=0_n$ and inserting $\hat{\omega}_i = 0_{n_i}\; i=1,\dots,n$ yields $\eta = k_1 1_n$, which inserted in \eqref{eq:hvdc_coordinated_droop_control_secondary_distributed_network} yields 
\begin{align}
P^\text{gen}_{i_k} &= k \frac{K^V_{i}}{K^\omega_{i_1}} K^\text{droop, I}_{i_k}, \; k=1, \dots, n_i. 
\label{eq:hvdc_coordinated_generation_ss_ac_network}
\end{align}
Finally we let $\dot{\hat{V}}=0_n$, insert $\tilde{\omega}=0_n$ and premultiply the equation with $1_n^TC$ and obtain 
\begin{align}
1_n^T K^V \hat{V} = 0.
\label{eq:hvdc_coordinated_voltages_ss_ac_network}
\end{align}
By similar arguments as in the proof of Corollary~\ref{cor:hvdc_coordinated_equilibrium}, we can show that \eqref{eq:hvdc_coordinated_generation_ss_ac_network} and \eqref{eq:hvdc_coordinated_voltages_ss_ac_network} are equivalent to the KKT conditions of \eqref{eq:hvdc_coordinated_generation_objective_ac_network}  and \eqref{eq:hvdc_coordinated_voltage_objective_ac_network}, respectively. This concludes the proof. 
\end{proof}
%
%
%
\section{Generalization to $\pi$-Link HVDC model}
\label{sec:PI-model}
In this section we extend the HVDC line model to consider the inductance and capacitance of the HVDC lines. We model the HVDC lines as series of $\ell$ $\pi$-links consisting of resistors, inductors, and capacitors. The dynamics of HVDC line $k$, connecting  converters $i$ and $j$, are given by
\begin{align}
C_i \dot{{V}}_i &= I_i^\text{inj} - \sum_{k\in \mathcal{N}_{i}^{\text{in}}} I_{k,1} \nonumber \\
L_k \dot{I}_{k,1} &= -R_k I_{k,1} + V_i - V_{k,1} \nonumber\\
C_k^{\text{line}} \dot{V}_{k,1} &= I_{k,1} - I_{k,2}\nonumber \\
L_k \dot{I}_{k,2} &= -R_k I_{k,2} + V_{k,1} - V_{k,2} \nonumber\\
&\;\;\vdots \nonumber\\
C_k^{\text{line}} \dot{V}_{k,\ell-1} &= I_{k,\ell-1} - I_{k,\ell} \nonumber \\
L_k \dot{I}_{k,\ell} &= -R_k I_{k,\ell} + V_{k,\ell-1} -V_j \nonumber\\
C_j \dot{\hat{V}}_j &= I^\text{inj}_j + \sum_{k\in \mathcal{N}_{i}^{\text{out}}} I_{k,\ell},
\label{eq:line_pi_scalar}
\end{align}
where $V_{k,q}$ and $I_{k,q}$ denote the DC voltage and current of line segment $q$, respectively, and $C^{\text{line}}_k$, $R_k$, and $L_k$ are the capacitance, resistance and inductance of each line segment of line $k$, respectively. The sets $\mathcal{N}_{i}^{\text{in}}$ and $\mathcal{N}_{i}^{\text{out}}$ denote the incoming and outgoing HVDC lines to converter $i$, respectively. 
%
%
%
%
%
To simplify the derivations we only consider AC areas consisting of single generators, but the results can be generalized to also include the AC generator network model of Section~\ref{sec:AC_network}. 
Combining the DC voltage dynamics \eqref{eq:line_pi_scalar}, the frequency dynamics \eqref{eq:frequency} with the generation control \eqref{eq:hvdc_coordinated_droop_control_secondary_distributed}, the converter controller \eqref{eq:voltage_control_secondary} with the power-current relationship \eqref{eq:power-current}, with $\phi''$ defined as in Section~\ref{sec:secondary_frequency_control} we obtain the closed-loop dynamics 
\begin{IEEEeqnarray}{rcl}
\dot{\hat{\omega}} &=& M \Big(- (K^{\text{droop}} + K^\omega) \hat{\omega} + K^V \hat{V} \nonumber \\
&& - {K^V}(K^\omega)^{-1} K^\text{droop, I} \eta - \mathcal{L}_\phi S \phi''  + P^{{m}}  \Big) \nonumber \\
\dot{\hat{V}} &=& 
 \frac{E}{V^\text{nom}} (- K^V \hat{V} -V^\text{nom} D_\text{in} I_1 + V^\text{nom} D_\text{out}I_\ell \nonumber \\
 && + {K}^\omega \hat{\omega} +  \mathcal{L}_\phi S \phi''  ) \nonumber \\
 \dot{I}_1 &=& L^{-1}(-RI_1+D_\text{in}^T \hat{V} - V_1) \nonumber\\
\dot{V}_1 &=&  E^{\text{line}}(I_1-I_2) \nonumber\\
\dot{I}_2 &=& L^{-1}(-RI_2 + V_1 - V_2) \nonumber\\
&\vdots& \nonumber\\
\dot{V}_{\ell-1} &=& E^\text{line} (I_{\ell-1} - I_\ell) \nonumber \\
\dot{I}_\ell &=& L^{-1}(-RI_\ell + V_{\ell-1} - D_\text{out} \hat{V}) \nonumber\\
\dot{\eta} &=& K^{\text{droop, I}}\hat{\omega} - \mathcal{L}_\eta \eta \nonumber \\
\dot{\phi}'' &=&  S^T(K^V)^{-1} {K}^\omega \hat{\omega}  -\gamma I_{n-1} \phi'',
\label{eq:hvdc_coordinated_voltage_control_pi_link_cl_dynamics}
\end{IEEEeqnarray}
 where 
$M$, $E$ $K^{\text{droop}}$, $K^\omega$, $K^V$, $K^\text{droop, I}$, $\mathcal{L}_\phi$, $\mathcal{L}_R$, $\mathcal{L}_\eta$, $P^m$, $\hat{\omega}$, $\hat{V}$ are defined as in Section~\ref{sec:secondary_frequency_control}, 
 $E^\text{line}=\diag((C^{\text{line}}_1)^{-1}, \dots, (C^{\text{line}}_m)^{-1})$, $L=\diag(L_1, \dots, L_m)$, and $R=\diag(R_1, \dots, R_m)$. The $n\times m$-matrices $D_\text{in}$ and $D_\text{out}$ describe the adjacency relations between the in- and outgoing HVDC lines and the converters. Element $(i,j)$ of $D_\text{in}$ or $D_\text{out}$ is $1$ if line $j$ originates or terminates at converter $i$, and $0$ otherwise.
The variables $I_q=[I_{1,q}-I_{1}^\text{nom}, \dots, I_{m,q}-I_{m}^\text{nom}]^T$ and $V_q=[V_{1,q}-V_{1,q}^\text{nom}, \dots, V_{m,q}-V_{m,q}^\text{nom}]^T$ define the incremental currents and voltages of the $q$:th line segment for all HVDC lines $1, \dots, m$. The constant $I_{k}^\text{nom}$ denotes the nominal current of each line segment of line $k$, and $V_{k,q}^\text{nom}$ denotes the nominal voltage of segment $q$ of line $k$. 
\begin{theorem}
The equilibrium of the dynamics \eqref{eq:hvdc_coordinated_voltage_control_pi_link_cl_dynamics} is globally asymptotically stable under Assumptions \ref{ass:L_phi_coordinated} and \ref{ass:gamma_coordinated}. 
\label{th:hvdc_coordinated_voltage_control_pi_link_stability}
\end{theorem}
\section{Simulations}
In this section, simulations are conducted on a test system to validate the performance of the proposed controllers. The simulation was performed in Matlab, using a dynamic phasor approach based on \cite{demiray2008}. The test system is illustrated in Fig.~\ref{fig:TestGrid}.
\label{sec:simulations}
\begin{figure}[htb]	
	\centering
	\tikzsetnextfilename{IEEE14Grid6_MTDC}
	\begin{tikzpicture}[>=triangle 45,scale=0.23,font=\small ]
	\pgfdeclarelayer{background}
	\pgfdeclarelayer{foreground}
	\pgfsetlayers{background,main,foreground}
	\draw (6,14) node(bus1) [busbarX2,cap=round] {};
	\draw (9,7) node(bus2) [busbarX4] {};
	\draw (30.2,7) node(bus3) [busbarX3] {};
	\draw (31,10) node(bus4) [busbarX4] {};
	\draw (14,13) node(bus5) [busbarX4] {};
	\draw (15.6,17) node(bus6) [busbarX4] {};
	\draw (31.8,14) node(bus7) [busbarX1point] {};
	\draw (35,14) node(bus8) [busbarX1vert] {};
	\draw (31.8,18) node(bus9) [busbarX3] {};
	\draw (26,20) node(bus10) [busbarX3] {};
	\draw (21,21) node(bus11) [busbarX3] {};
	\draw (8,23) node(bus12) [busbarX3] {};
	\draw (14.75,23) node(bus13) [busbarX3] {};
	\draw (28,23) node(bus14) [busbarX3] {};
	\draw (4.1,15.8) node(vscAC) [vsc2,rotate=180] {};
	\draw[pin-,line width=0.8pt, cap=round](bus1.pin1) -- ++(0, 1);
	\draw[line width=2pt,Blue](3.25,16.2) -- ++(-1.2, 0);
	\draw (bus1.pin2) node(gen1) [generator] {};
	\draw (bus2.pin2) node(gen2) [generator] {};
	\draw (bus3.pin3) node(gen3) [generator_] {};
	\draw (11,14) node(gen4) [generator_liegend] {};
	\draw (35.5,14) node(gen5) [generator_] {};
	\draw[pin-,line width=0.8pt, cap=round] (bus6.pin1) -- ++(0, -1)-- (gen4.right);
	\draw[pin-,line width=0.8pt, cap=round] (bus8.pin1) -- ++(0.5, 0);
	\draw (bus2.pin2) node (load2) [load] {};
	\draw (bus3.pin2) node (load3) [load] {};
	\draw (bus4.pin4) node (load4) [load] {};
	\draw (bus5.pin3) node (load5) [load] {};
	\draw (bus6.pin4) node (load6) [load] {};
	\draw (bus9.pin3) node (load9) [load] {};
	\draw (bus10.pin2) node (load10) [load] {};
	\draw (bus11.pin2) node (load11) [load] {};
	\draw (bus12.pin1) node (load12) [load] {};
	\draw (bus13.pin3) node (load13) [load] {};
	\draw (bus14.pin2) node (load14) [load] {};
	\draw (14.8,15)node(trafo1) [trafo, scale=0.5,rotate=90] {};
	\draw (30.2,14)node(trafo2) [trafo, scale=0.5,rotate=90] {};
	\draw (31.8,12)node(trafo3) [trafo, scale=0.5,rotate=90] {};
	\draw (31.8,16)node(trafo4) [trafo, scale=0.5,rotate=90] {};
	\draw (33.5,14)node(trafo5) [trafo, scale=0.5] {};
	\draw[pin-pin] (bus1.pin1) -- ++(0, -1.5) --($(bus2.pin1)+(0, 1.5)$) -- (bus2.pin1);
	\draw[pin-pin] (bus1.pin2) -- ++(0, -1.5) --($(bus5.pin1)+(0, -1.5)$) -- (bus5.pin1);
	\draw[pin-pin] (bus2.pin4) -- ++(0, -1.5) --($(bus3.pin1)+(0, -1.5)$) -- (bus3.pin1);
	\draw[pin-pin] (bus2.pin4) -- ++(0, 1.5) --($(bus4.pin2)+(0, -1.5)$) -- (bus4.pin2);
	\draw[pin-pin] (bus2.pin3) -- ++(0, 1.5) --($(bus5.pin2)+(0, -1.5)$) -- (bus5.pin2);
	\draw[pin-pin] (bus3.pin3) -- ++(0, 1.5) --($(bus4.pin3)+(0, -1.5)$) -- (bus4.pin3);
	\draw[pin-pin] (bus4.pin1) -- ++(0, -1.0) -| (bus5.pin4);
	\draw[pin-] (bus5.pin3) -- (trafo1.left);
	\draw[pin-] (bus6.pin2) -- (trafo1.right);
	\draw[pin-] (bus4.pin2) -- (trafo2.left);
	\draw[pin-] (bus9.pin1) -- (trafo2.right);
	\draw[pin-] (bus4.pin3) -- (trafo3.left);
	\draw[pin-] (bus7.pin1) -- (trafo3.right);
	\draw[pin-] (bus7.pin1) -- (trafo4.left);
	\draw[pin-] (bus9.pin2) -- (trafo4.right);
	\draw[pin-] (bus7.pin1) -- (trafo5.left);
	\draw[pin-] (bus8.pin1) -- (trafo5.right);
	\draw[pin-pin] (bus6.pin3) -- ++(0, 1.5) -- ($(bus11.pin1)+(0, -1.5)$) --(bus11.pin1);
	\draw[pin-pin] (bus6.pin1) -- ++(0, 1.5) -- ($(bus12.pin2)+(0, -1.5)$) --(bus12.pin2);
	\draw[pin-pin] (bus6.pin2) -- (bus13.pin2);
	\draw[pin-pin] (bus6.pin1) -- ++(0, 1.5) -- ($(bus12.pin2)+(0, -1.5)$) --(bus12.pin2);
	\draw[pin-pin] (bus9.pin1) -- ++(0, 1) -- ($(bus10.pin3)+(0, -1)$) --(bus10.pin3);
	\draw[pin-pin] (bus9.pin3) -- ++(0, 2.5) -- ($(bus14.pin3)+(0, -1.5)$) --(bus14.pin3);
	\draw[pin-pin] (bus10.pin1) -- ++(0, -1.5) -- ($(bus11.pin3)+(0, -1.5)$) --(bus11.pin3);
	\draw[pin-pin] (bus12.pin3) -- ++(0, -1.5) -- ($(bus13.pin1)+(0, -1.5)$) --(bus13.pin1);
	\draw[pin-pin] (bus13.pin3) -- ++(0, 1) -- ($(bus14.pin1)+(0, 1)$) --(bus14.pin1);
	\draw[line width=1pt] (6,3.5)-- (33,3.5) -- (37,10.5) -- (37,17.5) -- (33,24.5)--(6,24.5) -- (2,17.5)--(2,10.5)--(6,3.5); 
	\draw[line width=0.8pt](	3	,	37.75	)--(	9.75	,	37.75	)--(	10.75	,	39.5	)--(	10.75	,	41.25	)--(	9.75	,	43	)--(	3	,	43	)--(	2	,	41.25	)--(	2	,	39.5	)--(	3	,	37.75	);  
	\draw[line width=0.8pt,fill=white](	16	,	37.75	)--(	22.75	,	37.75	)--(	23.75	,	39.5	)--(	23.75	,	41.25	)--(	22.75	,	43	)--(	16	,	43	)--(	15	,	41.25	)--(	15	,	39.5	)--(	16	,	37.75	);  
	\draw[line width=0.8pt,fill=white](3,25)--(	9.75	,	25	)--(	10.75	,	26.75	)--(	10.75	,	28.5	)--(	9.75	,	30.25	)--(	3	,	30.25	)--(	2	,	28.5	)--(	2	,	26.75	)--(	3	,	25	); 
	\draw[line width=.8pt,fill=white](16,25)--(22.75,25)--(23.75,26.75)--(23.75,28.5)--(22.75,30.25)--(16,30.25)--(15,28.5)--(15,26.75)--(16,25); 
	\draw[line width=0.8pt,fill=white](29.25,25)--(36,25)--(37,26.75)--(37,28.5)--(36,30.25)--(29.25,30.25)--(28.25,28.5)--(28.25,26.75)--(29.25,25); 
	\draw[line width=0.8pt,fill=white](	29.25	,	37.75	)--(	36	,	37.75	)--(	37	,	39.5	)--(	37	,	41.25	)--(	36	,	43	)--(	29.25	,	43	)--(	28.25	,	41.25	)--(	28.25	,	39.5	)--(	29.25	,	37.75	); 
	\draw (5.4,37) node(busDC1) [busbarDCX2,cap=round] {};
	\draw (19.5,37) node(busDC2) [busbarDCX5,cap=round] {};
	\draw (7,31) node(busDC3) [busbarDCX4,cap=round] {};
	\draw (19.5,31) node(busDC4) [busbarDCX3,cap=round] {};
	\draw (32,31) node(busDC5) [busbarDCX4,cap=round] {};
	\draw (33.6,37) node(busDC6) [busbarDCX2,cap=round] {};
	\draw (6.5,39.05) node(vsc1) [vsc2,rotate=-90] {};   
	\draw[pinDC-,Blue,line width=1pt] (busDC1.pin2) -- ++(0, 1.2);
	\draw[line width=1pt] ($(vsc1)+(0.8, 0.5)$) -- ++(1,0);
	\draw[line width=1pt,dotted] ($(vsc1)+(1.8, 0.5)$) -- ++(1.5,0);
	\draw (19.5,39.05) node(vsc2) [vsc2,rotate=-90] {};
	\draw[pinDC-,Blue,line width=1pt] (busDC2.pin3) -- ++(0, 1.2);
	\draw[line width=1pt] ($(vsc2)+(0.8, 0.5)$) -- ++(1,0);
	\draw[line width=1pt,dotted] ($(vsc2)+(1.8, 0.5)$) -- ++(1.5,0);
	\draw (6,28.95) node(vsc3) [vsc2,rotate=90] {};
	\draw[pinDC-,Blue,line width=1pt] (busDC3.pin2) -- ++(0, -1.2);
	\draw[line width=1pt] ($(vsc3)+(-0.8, -0.4)$) -- ++(-1,0);
	\draw[line width=1pt,dotted] ($(vsc3)+(-1.8, -0.4)$) -- ++(-1.5,0);
	\draw (19,28.95) node(vsc4) [vsc2,rotate=90] {};
	\draw[pinDC-,Blue,line width=1pt] (busDC4.pin2) -- ++(0, -1.2);
	\draw[line width=1pt] ($(vsc4)+(-0.8, -0.4)$) -- ++(-1,0);
	\draw[line width=1pt,dotted] ($(vsc4)+(-1.8, -0.4)$) -- ++(-1.5,0);
	\draw (32.25,28.95) node(vsc5) [vsc2,rotate=90] {};
	\draw[pinDC-,Blue,line width=1pt] (busDC5.pin3) -- ++(0, -1.2);
	\draw[line width=1pt] ($(vsc5)+(-0.8, -0.4)$) -- ++(-1,0);
	\draw[line width=1pt,dotted] ($(vsc5)+(-1.8, -0.4)$) -- ++(-1.5,0);
	\draw (32.75,39.05) node(vsc6) [vsc2,rotate=-90] {};
	\draw[pinDC-,Blue,line width=1pt] (busDC6.pin1) -- ++(0, 1.2);
	\draw[line width=1pt] ($(vsc6)+(0.8, 0.5)$) -- ++(1,0);
	\draw[line width=1pt,dotted] ($(vsc6)+(1.8, 0.5)$) -- ++(1.5,0);
	\draw[pinDC-pinDC,Blue,line width=2pt] (busDC1.pin2) -- ++(0, -1)-|(busDC2.pin1);
	\draw[pinDC-pinDC,Blue,line width=2pt] (busDC1.pin1) -|(busDC3.pin1);
	\draw[pinDC-pinDC,Blue,line width=2pt] (busDC2.pin2) -- ++(0, -1.5)--($(busDC3.pin2)+(0, 2.5)$)--(busDC3.pin2);
	\draw[pinDC-pinDC,Blue,line width=2pt] (busDC2.pin3) -|(busDC4.pin2);
	\draw[pinDC-pinDC,Blue,line width=2pt] (busDC2.pin4) -- ++(0, -1.5)--($(busDC5.pin3)+(0, 2.5)$)--(busDC5.pin3);
	\draw[pinDC-pinDC,Blue,line width=2pt] (busDC2.pin5) -- ++(0, -1)-|(busDC6.pin1);
	\draw[pinDC-pinDC,Blue,line width=2pt] (busDC3.pin4) -- ++(0, 1)-|(busDC4.pin1);
	\draw[pinDC-pinDC,Blue,line width=2pt] (busDC3.pin3) -- ++(0, 2)--($(busDC5.pin2)+(0, 2)$)--(busDC5.pin2);
	\draw[pinDC-pinDC,Blue,line width=2pt] (busDC4.pin3) -- ++(0, 1)--($(busDC5.pin1)+(0, 1)$)--(busDC5.pin1);
	\draw[pinDC-pinDC,Blue,line width=2pt] (busDC5.pin4) -- ++(0, 1.5)--($(busDC6.pin2)+(0, -1.5)$)--(busDC6.pin2);
	\draw[OliveGreen,line width=2pt,dashed] ($(busDC1.pin2)+(0.4, -0.1)$) -- ++(0, -0.5)-|($(busDC2.pin1)+(-0.4,-0.1)$);
	\draw[OliveGreen,line width=2pt,dashed] ($(busDC2.pin2)+(-0.4, -0.1)$) -- ++(0, -1)--($(busDC3.pin2)+(-0.4,3)$)--($(busDC3.pin2)+(-0.4,0.1)$);
	\draw[OliveGreen,line width=2pt,dashed] ($(busDC3.pin4)+(0.4, 0.1)$) -- ++(0, 0.5)-|($(busDC4.pin1)+(-0.4,0.1)$);
	\draw[OliveGreen,line width=2pt,dashed] ($(busDC4.pin3)+(0.4, 0.1)$) -- ++(0, 0.5)-|($(busDC5.pin1)+(-0.4,0.1)$);
	\draw[OliveGreen,line width=2pt,dashed] ($(busDC5.pin4)+(0.4, 0.1)$) -|($(busDC6.pin2)+(0.4,-0.1)$);
	{
		\draw (bus1.pin1) node[left] {1};
		\draw (bus2.pin1) node[left] {2};
		\draw (bus3.pin1) node[left] {3};
		\draw (bus4.pin1) node[above left] {4};
		\draw (bus5.pin1) node[left] {5};
		\draw (bus6.pin4) node[right] {6};1
		\draw (bus7.pin1) node[below right=0.1 and 0.1] {7};
		\draw (bus8.pin1) node[above right] {8};
		\draw (bus9.pin3) node[right] {3};
		\draw (bus10.pin1) node[above] {10};
		\draw (bus11.pin1) node[left] {11};
		\draw (bus12.pin3) node[right] {12};
		\draw (bus13.pin3) node[right] {13};
		\draw (bus14.pin3) node[right] {14};
		\node  (ACAreaName) at (22,12) {\large{AC area 4}};
		\node  (ACArea1Name) at (6.2,42) {AC Area 1};
		\node  (ACbus1Name) at (9,40.5) {bus 1};
		\node  (ACArea2Name) at (19.2,42) {AC Area 2};
		\node  (ACbus2Name) at (22,40.5) {bus 2};	
		\node  (ACArea3Name) at (6.2,26) {AC Area 3};
		\node  (ACbus3Name) at (3.55,27.5) {bus 4};		
		\node  (ACArea4Name) at (19.2,26) {AC Area 4};
		\node  (ACbus4Name) at (16.55,27.5) {bus 1};
		\node  (ACArea5Name) at (32.45,26) {AC Area 5};
		\node  (ACbus5Name) at (29.8,27.5) {bus 5};
		\node  (ACArea6Name) at (32.45,42) {AC Area 6};
		\node  (ACbus6Name) at (35.25,40.5) {bus 3};
		\node  (DCAreaName) at (18.6,33.85) {\large\color{Blue}{MTDC Grid}};								
	}
	 \begin{pgfonlayer}{background}
	 \shade[top color=white!90!black,bottom color=black!60!white]  (22.75,30.3) -- (33.13,24.5) -- (5.88,24.5) -- (16,30.3);
	  \end{pgfonlayer}
	\end{tikzpicture}
	\caption{Test grid consisting of 6 asynchronous IEEE 14 AC grids connected through an MTDC grid.}
	\label{fig:TestGrid}
	\vspace{-0.2cm}
\end{figure}
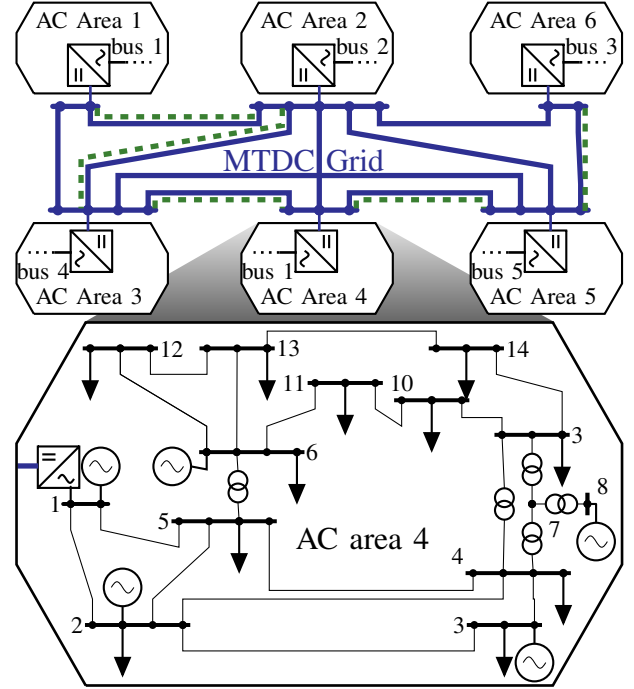%
We model the HVDC line as a single $\pi$-link, with parameters given in Table~\ref{tab:HVDCgridParameter}. The terminal capacitances  are $C_i=0.375\times 10^{-3}$ p.u. The AC grid parameters were obtained from \cite{Milano}. The generators are modeled as a $6$th order synchronous machine model controlled by an automatic voltage controller and a governor \cite{kundur1994power}. The consumers in all AC grids are modeled as constant power loads. It is assumed the linear power-current relation \eqref{eq:power-current_nonlinear} holds. 
\begin{table}[htb]
	\centering
	\caption{HVDC grid line parameters}
	\vspace{-0.2cm}
	\label{tab:HVDCgridParameter}
	\begin{tabular}{lllll} \toprule
		$(i,j)$&$R_{ij}$ [p.u.]&$L_{ij}$ [$10^{-3}$ p.u.]& $C_{ij}$ [p.u.]\\ \midrule
		(1, 2), (1, 3), (2,4), (3,4) & 0.0586 & 0.2560 &0.0085 \\
		(2, 3) & 0.0878 & 0.3840 &0.0127\\
		(2, 5), (4,5) & 0.0732 & 0.3200 &0.0106\\
		(2, 6), (3, 5), (5, 6) & 0.1464 & 0.6400 &0.0212\\
		\bottomrule
	\end{tabular}
\end{table}
\begin{table}[htb]
                \centering
				 \caption{Controller parameters}
			 	\vspace{-0.2cm}
                \label{tab:ControllerParameter}
                \begin{tabular}{@{}lllllll@{}}\toprule
                               $K^{\omega}_i$ &$K^{V}_i$&$K^\text{droop}_i$& $K^\text{droop, I}_i$ &$c^{\eta}_{ij}$ &$c^{\phi}_{ij}$ & $\gamma$ \\
                               \midrule
                               1501            &   80       & 9   & 3.35      &  ${5}/{R_{ij}}$ & ${15}/{R_{ij}}$  &  0\\
                               \bottomrule
                \end{tabular}
\end{table}

The three different controllers proposed in this paper are applied to the test grid, i.e.,  \eqref{eq:hvdc_coordinated_droop_control_secondary_distributed_decentralized_version} and \eqref{eq:voltage_control_secondary_decentralized_version}, \eqref{eq:hvdc_coordinated_droop_control_secondary_distributed} and \eqref{eq:voltage_control_secondary_decentralized_version}, and
\eqref{eq:hvdc_coordinated_droop_control_secondary_distributed}  and \eqref{eq:voltage_control_secondary}, with parameters given in Table~\ref{tab:ControllerParameter}.
The communication network of  \eqref{eq:hvdc_coordinated_droop_control_secondary_distributed} and \eqref{eq:voltage_control_secondary} is illustrated by the dashed lines in Fig.~\ref{fig:TestGrid}. 

We set $\gamma=0$, so Theorem~\ref{th:hvdc_coordinated_voltage_control_secondary_network_stability} does not guarantee stability of the equilibrium. However, the closed-loop system matrix can easily be verified to be Hurwitz. At time $t=1$ the output of  generator 2 in AC area 1 was reduced by $0.2$ p.u.
Fig.~\ref{fig:ACFrequency} shows the average frequencies of the AC grids for all three controllers. 
Immediately after the fault the frequencies of the AC area of the increased load drop. The frequency drop is followed by a DC voltage drop in all converter nodes, due to  \eqref{eq:voltage_control_secondary} or \eqref{eq:voltage_control_secondary_decentralized_version}.  
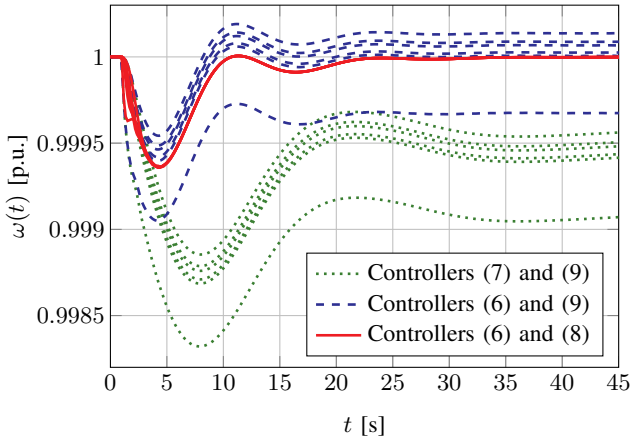
\begin{figure}[b]
	\centering
	\tikzsetnextfilename{Frequencies}
	\begin{tikzpicture}[font=\small]
	\begin{axis}
	[cycle list name=Necsys,
	xlabel={$t$ [s]},
	ylabel={$\omega(t)$ [p.u.]},
	xmin=0,
	xmax=45,
	xtick={0,5,...,45},
	ymin=0.9982,
	ymax=1.0003,
	ytick={0.9980,0.9985,0.9990,0.9995,1,1.0005},
	yticklabel style={/pgf/number format/.cd,
		fixed,
		precision=4},
	grid=major,
	height=6.4cm,
	width=0.96\columnwidth,
	legend cell align=left,
	legend pos= south east,
	legend entries={Controllers (7) and (9), Controllers (6) and (9), Controllers (6) and (8)
	},
	]
	\addlegendimage{no markers,OliveGreen, dotted, line width=1pt}
	\addlegendimage{no markers,Blue, dashed, line width=1pt}
	\addlegendimage{no markers,Red, line width=1pt}
	\foreach \x in {1,2,3,4,5,6}{
		\addplot+[dotted,OliveGreen] table[x index=0,y index=\x,col sep=tab]{PlotData_coordinated/ACC_Frequency.txt};
	}
	\pgfplotsset{cycle list shift=1};
	\foreach \x in {1,2,3,4,5,6}{
		\addplot+[dashed,Blue] table[x index=0,y index=\x,col sep=tab]{PlotData_coordinated/CDC_Frequency.txt};
	}	
	\pgfplotsset{cycle list shift=2};
	\foreach \x in {1,2,3,4,5,6}{
		\addplot+ [Red] table[x index=0,y index=\x,col sep=tab]{PlotData_coordinated/Necsys_Frequency.txt};
	}
	\end{axis}
	\end{tikzpicture}
	\caption{Average AC area frequencies. After a reduced generation of $0.2$ p.u. at $t=1$~s, the frequencies synchronize fast, and are subsequently restored to the nominal frequency when the distributed generator and converter controllers \eqref{eq:hvdc_coordinated_droop_control_secondary_distributed} and  \eqref{eq:voltage_control_secondary} are employed. For the controllers 
		\eqref{eq:hvdc_coordinated_droop_control_secondary_distributed} and 	\eqref{eq:voltage_control_secondary_decentralized_version} or  \eqref{eq:hvdc_coordinated_droop_control_secondary_distributed_decentralized_version} and \eqref{eq:voltage_control_secondary_decentralized_version}, a static control error remains present. }
	\label{fig:ACFrequency}
\end{figure}
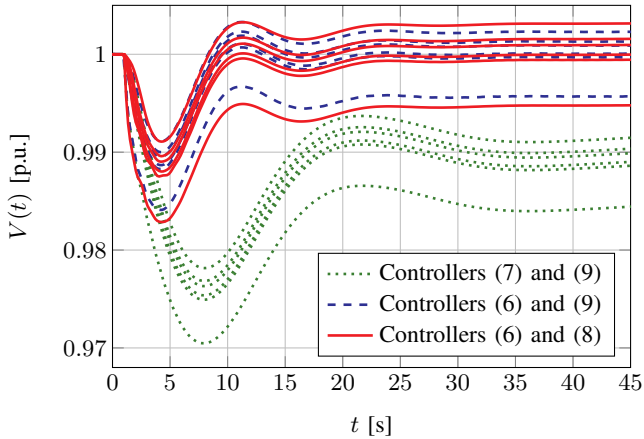
\begin{figure}[b!]
	\centering
	\tikzsetnextfilename{DCVoltages}
	\begin{tikzpicture}[font=\small]
	\begin{axis}[cycle list name=Necsys,
	xlabel={$t$ [s]},
	ylabel={$V(t)$ [p.u.]},
	xmin=0,
	xmax=45,
	xtick={0,5,...,45}, 
	ymin=0.968,
	ymax=1.005,
	yticklabel style={/pgf/number format/.cd,
		fixed,
		precision=3},
	grid=major,
	height=6.4cm,
	width=0.975\columnwidth,
	legend cell align=left,
	legend pos= south east,
	legend entries={Controllers (7) and (9), Controllers (6) and (9), Controllers (6) and (8)
	},
	]
	\addlegendimage{no markers,OliveGreen, dotted, line width=1pt}
	\addlegendimage{no markers,Blue, dashed, line width=1pt}
	\addlegendimage{no markers,Red, line width=1pt}
	\foreach \x in {1,2,3,4,5,6}{
		\addplot+[dotted,OliveGreen] table[x index=0,y index=\x,col sep=tab]{PlotData_coordinated/ACC_Voltage.txt};
	}
	\pgfplotsset{cycle list shift=1};
	\foreach \x in {1,2,3,4,5,6}{
		\addplot+[dashed,Blue] table[x index=0,y index=\x,col sep=tab]{PlotData_coordinated/CDC_Voltage.txt};
	}	
	\pgfplotsset{cycle list shift=2};
	\foreach \x in {1,2,3,4,5,6}{
		\addplot+ [Red] table[x index=0,y index=\x,col sep=tab]{PlotData_coordinated/Necsys_Voltage.txt};
	}
	\end{axis}
	\end{tikzpicture}
	\caption{DC converter voltages. After a reduced generation of $0.2$ p.u. at $t=1$~s, an immediate voltage drop is followed by a restoration of the average voltage errors to zero when the distributed generator controller \eqref{eq:hvdc_coordinated_droop_control_secondary_distributed} is employed together with either converter controller  \eqref{eq:voltage_control_secondary} or \eqref{eq:voltage_control_secondary_decentralized_version}. This corresponds to the minimization of \eqref{eq:hvdc_coordinated_voltage_objective}. However, when  \eqref{eq:hvdc_coordinated_droop_control_secondary_distributed_decentralized_version} is employed, the average voltage error does not converge to zero. }
	\label{fig:DCVoltage}
\end{figure}
We note that the frequencies are restored to the nominal frequency for \eqref{eq:hvdc_coordinated_droop_control_secondary_distributed}  and \eqref{eq:voltage_control_secondary} as predicted by \eqref{eq:hvdc_coordinated_frequency_objective_ac_network}. However, for the controller combination \eqref{eq:hvdc_coordinated_droop_control_secondary_distributed}  and \eqref{eq:voltage_control_secondary_decentralized_version} the frequencies are not restored to the nominal values, despite that a secondary frequency controller is employed. An intuitive explanation to this is that the distributed frequency controller \eqref{eq:hvdc_coordinated_droop_control_secondary_distributed} requires that the frequencies in the different areas synchronize. This synchronization is achieved by \eqref{eq:voltage_control_secondary}, but not by \eqref{eq:voltage_control_secondary_decentralized_version}. 

Fig.~\ref{fig:DCVoltage} shows the DC voltages of the terminals. A similar behavior as from the frequencies can be seen. As foreseen, the voltage deviations can not be controlled back to zero, since this would result in zero current flows. For the controllers \eqref{eq:hvdc_coordinated_droop_control_secondary_distributed}  and \eqref{eq:voltage_control_secondary}, the average voltage deviation is restored to zero as predicted by \eqref{eq:hvdc_coordinated_voltages_ss_ac_network}. 

Fig.~\ref{fig:Generators} shows the total increase of the generated power within each AC area. Initially all generators have a similar oscillating behavior, due to the AC voltage oscillations. 
For the controllers \eqref{eq:hvdc_coordinated_droop_control_secondary_distributed}  and \eqref{eq:voltage_control_secondary} the total increase in the generated power of the AC areas converge to the same value, thus minimizing \eqref{eq:hvdc_coordinated_generation_objective_ac_network}, since the controller parameters are chosen uniformly for all converters $i$. For the other two controller combinations, the increase in the generated power of the AC areas does not converge to the same value. This implies that \eqref{eq:hvdc_coordinated_generation_objective_ac_network} is not minimized, and we conclude that the simulations confirm that the controllers \eqref{eq:hvdc_coordinated_droop_control_secondary_distributed}  and \eqref{eq:voltage_control_secondary} have superior power sharing properties. 
Clearly, the communication in the controllers  \eqref{eq:hvdc_coordinated_droop_control_secondary_distributed} and \eqref{eq:voltage_control_secondary} is essential to eliminate the frequency error and sharing frequency control reserves. 
We can note relatively high frequency oscillations in the transient response of the generated power. These oscillations originate from the averaging term in the dynamics of the internal controller variables $\eta_i$. By reducing the constants $c_{ij}^\eta$, these oscillations can be reduced at the expense of slower convergence.

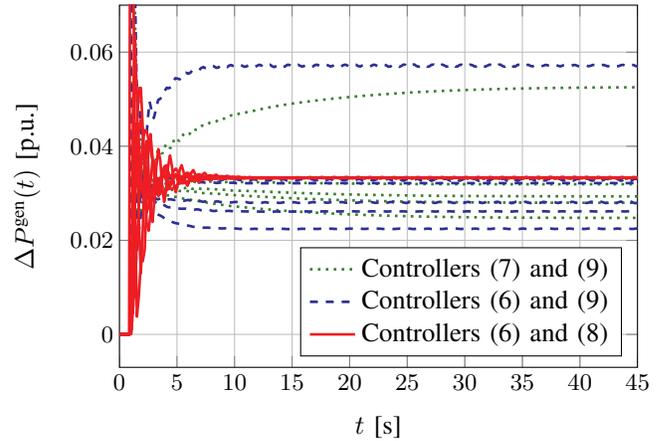
\begin{figure}[t!]
	\centering
	\tikzsetnextfilename{Generators}
	\begin{tikzpicture}
	\begin{axis}[cycle list name=Necsys,
	xlabel={$t$ [s]},
	ylabel={$\Delta P^\text{gen}(t)$ [p.u.]},
	xmin=0,
	xmax=45,
	xtick={0,5,...,45},
	ymin=-0.007,
	ymax=0.070,
	scaled ticks=false, 
	tick label style={/pgf/number format/fixed,
		/pgf/number format/precision=2},
	ytick={0,0.02,...,0.101},
	grid=major,
	height=6.4cm,
	width=0.975\columnwidth,
	legend cell align=left,
	legend pos= south east,
	legend entries={Controllers (7) and (9), Controllers (6) and (9), Controllers (6) and (8)
	},
	]
	\addlegendimage{no markers,OliveGreen, dotted, line width=1pt}
	\addlegendimage{no markers,Blue, dashed, line width=1pt}
	\addlegendimage{no markers,Red, line width=1pt}
	\foreach \x in {1,2,3,4,5,6}{
		\addplot+[dotted,OliveGreen,line join=round] table[x index=0,y index=\x,col sep=tab]{PlotData_coordinated/ACC_Generators.txt};
	}
	\pgfplotsset{cycle list shift=1};
	\foreach \x in {1,2,3,4,5,6}{
		\addplot+[dashed,Blue,line join=round] table[x index=0,y index=\x,col sep=tab]{PlotData_coordinated/CDC_Generators.txt};
	}	
	\pgfplotsset{cycle list shift=2};
	\foreach \x in {1,2,3,4,5,6}{
		\addplot+ [Red,line join=round] table[x index=0,y index=\x,col sep=tab]{PlotData_coordinated/Necsys_Generators.txt};
	}
	\end{axis}
	\end{tikzpicture} 
	\caption{Total increase of generated power in the AC areas. After a reduced generation of $0.2$ p.u. at $t=1$~s. The generated power is increased for all AC areas regardless of the controllers employed. When the distributed generator and converter controllers \eqref{eq:hvdc_coordinated_droop_control_secondary_distributed} and  \eqref{eq:voltage_control_secondary} are employed, all generators generate to the same power and have the same marginal generation costs asymptotically, thus minimizing \eqref{eq:hvdc_coordinated_generation_objective} asymptotically. For \eqref{eq:hvdc_coordinated_droop_control_secondary_distributed} and \eqref{eq:voltage_control_secondary_decentralized_version} or \eqref{eq:hvdc_coordinated_droop_control_secondary_distributed_decentralized_version} and \eqref{eq:voltage_control_secondary_decentralized_version}, the marginal generation costs do not converge to the same value. }
	\label{fig:Generators}
\end{figure}

Fig.~\ref{fig:Converters} shows the power set-points of the converters. For controller \eqref{eq:hvdc_coordinated_droop_control_secondary_distributed}  and \eqref{eq:voltage_control_secondary} the control in area 1 overcompensate the disturbance at the beginning, also the other areas have some minor oscillations. A stable operation point is found after 6 s. Area 2-6 have all an equal share of the disturbance. It takes longer time for the other controllers to find a stable operation point, where controllers \eqref{eq:hvdc_coordinated_droop_control_secondary_distributed} and \eqref{eq:voltage_control_secondary_decentralized_version} are significantly faster than \eqref{eq:hvdc_coordinated_droop_control_secondary_distributed_decentralized_version} and \eqref{eq:voltage_control_secondary_decentralized_version}. The sharing of the disturbance are for both controller not equal, indicated by the different power set-points of the converters. 

\begin{figure}[t!]
	\centering
	\tikzsetnextfilename{ConverterPower}
	\begin{tikzpicture}
	\begin{axis}[cycle list name=Necsys,
	xlabel={$t$ [s]},
	ylabel={$P^\text{inj}_i(t)$ [p.u.]},
	xmin=0,
	xmax=45,
	xtick={0,5,...,45},
	ymin=-0.2,
	ymax=0.070,
	scaled ticks=false, 
	tick label style={/pgf/number format/fixed,
		/pgf/number format/precision=2},
	ytick={-0.2,-0.15,...,0.051},
	grid=major,
	height=6.4cm,
	width=0.975\columnwidth,
	legend cell align=left,
	legend style={at={(0.97,0.5)},anchor= east},
	legend entries={Controllers (7) and (9), Controllers (6) and (9), Controllers (6) and (8)
	},
	]
	\addlegendimage{no markers,OliveGreen, dotted, line width=1pt}
	\addlegendimage{no markers,Blue, dashed, line width=1pt}
	\addlegendimage{no markers,Red, line width=1pt}
	\foreach \x in {1,2,3,4,5,6}{
		\addplot+[dotted,OliveGreen,line join=round] table[x index=0,y index=\x,col sep=tab]{PlotData_coordinated/ACC_PSet.txt};
	}
	\pgfplotsset{cycle list shift=1};
	\foreach \x in {1,2,3,4,5,6}{
		\addplot+[dashed,Blue,line join=round] table[x index=0,y index=\x,col sep=tab]{PlotData_coordinated/CDC_PSet.txt};
	}	
	\pgfplotsset{cycle list shift=2};
	\foreach \x in {1,2,3,4,5,6}{
		\addplot+ [Red,line join=round] table[x index=0,y index=\x,col sep=tab]{PlotData_coordinated/Necsys_PSet.txt};
	}
	\end{axis}
\end{tikzpicture} 
	\caption{Converter power set-points for all 6 areas. After a reduced generation of $0.2$ p.u. at $t=1$~s. For all controllers the area with the disturbance starts to extract power from the MTDC grid and all others inject power. When the distributed generator and converter controllers \eqref{eq:hvdc_coordinated_droop_control_secondary_distributed} and \eqref{eq:voltage_control_secondary} are employed, the converter set-points stabilizing after $6$~s to a point with perfect sharing of the disturbance. For \eqref{eq:hvdc_coordinated_droop_control_secondary_distributed} and	\eqref{eq:voltage_control_secondary_decentralized_version} or  \eqref{eq:hvdc_coordinated_droop_control_secondary_distributed_decentralized_version} and \eqref{eq:voltage_control_secondary_decentralized_version}, it takes some more time to find a stable point and the converters differ from each other. }
	\label{fig:Converters}
\end{figure}
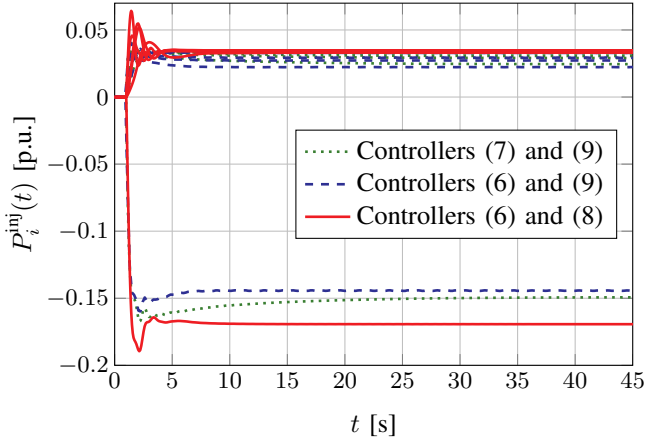
\section{Conclusions}
\label{sec:discussion}
In this paper we have studied distributed secondary controllers for sharing frequency control reserves of asynchronous AC systems connected through an MTDC system. The proposed controllers were shown to stabilize the interconnected AC systems and the MTDC grid. The AC grid frequencies were shown to converge to the nominal frequency. Furthermore, quadratic cost functions of the voltage deviations of the MTDC terminals and of the generated power was minimized asymptotically. The results were first derived for single-generator AC grids and purely resistive MTDC grids, and later generalized to AC grids of arbitrary size, and HVDC links modelled as $\pi$-links. Finally, the results were validated on a six-terminal MTDC system with connected IEEE 14 bus AC grids. 
Future work will focus on stability results under communication delays. 
\appendix
\begin{proof}[Proof of Theorem \ref{th:hvdc_coordinated_voltage_control_secondary_network_stability}]
Without loss of generality, let $P^m_i=0_{n_i}, \; i=1, \dots, n$. Consider the Lyapunov function candidate, which is positive definite and radially unbounded: 
\begin{align*}
W&= \sum_{i=1}^n \frac{K^\omega_{i_1}}{2K^V_i} \Bigg(  (\delta''_i)^T S_i^T \mathcal{L}^\text{AC}_{i} S_i \delta''_i  + \hat{\omega}_i^T M^{-1}_i \hat{\omega}_i \Bigg)   \nonumber \\
&\;\;\;\;  + \frac{V^\text{nom}}{2} \hat{V}^T C \hat{V} + \frac 12 {\eta}^T {\eta} + \frac 12 {\phi''}^T S^T \mathcal{L}_\phi S {\phi''}.
\end{align*}
Differentiating $W$ along trajectories of \eqref{eq:hvdc_coordinated_voltage_control_secondary_network_cl_dynamics} yields
\begin{align*}
\dot{W} 
&= \sum_{i=1}^n \frac{K^\omega_{i_1}}{K^V_i } \Bigg( (\delta''_i)^T S_i^T \mathcal{L}^\text{AC}_{i} S_i S_i^T \hat{\omega}_i \\
&\;\;\;\;  + \hat{\omega}_i^T \Big( - (K^{\text{droop}}_{i} + K^\omega_{i}) \hat{\omega}_{i} + e_1 K_i^V \hat{V}_i -\mathcal{L}^\text{AC}_{i} S_i
\delta_i''  \nonumber \\
& \;\;\;\; - \frac{K^V_{i}}{K^\omega_{i_1}} K^\text{droop, I}_{i} 1_{n_i} \eta_i - e_1 e_i^T \mathcal{L}_\phi S \phi'' \Big) \Bigg) \\
&\;\;\;\;  + \hat{V}^T \Big( \tilde{K}^\omega \tilde{\omega} -  \big(V^\text{nom} \mathcal{L}_R + {K^V} \big) \hat{V} + \mathcal{L}_\phi S \phi'' \Big) \\
&\;\;\;\; + \eta^T \Bigg( \sum_{i=1}^n e_i 1_n^T K_{i}^{\text{droop,I}}\hat{\omega}_{i} - \mathcal{L}_\eta \eta \Bigg) \\
&\;\;\;\; + {\phi''}^T S^T \mathcal{L}_\phi S \Big( S^T(K^V)^{-1}\tilde{K}^\omega \tilde{\omega}  -\gamma I_{n-1} \phi'' \Big)
\end{align*}
Since $SS^T = I_n - \frac 1n 1_{n\times n}$ and $S_iS_i^T = I_{n_i} - \frac {1}{n_i} 1_{n_i\times n_i}$, we have  $\mathcal{L}_\phi SS^T = \mathcal{L}_\phi$ and $\mathcal{L}_i^\text{AC} S_iS_i^T = \mathcal{L}_i^\text{AC}$ \cite{Andreasson2014TAC}. Furthermore 
\begin{IEEEeqnarray*}{rcl}
\sum_{i=1}^n \frac{K^\omega_{i_1}}{K^V_i } \hat{\omega}_i^T e_1 e_i^T \mathcal{L}_\phi S \phi'' & = & \tilde{\omega}^T \tilde{K}^\omega (K^V)^{-1} \mathcal{L}_\phi S \phi'' \\
\sum_{i=1}^n \frac{K^\omega_{i_1}}{K^V_i } \hat{\omega}_i^T e_1 K^V_i \hat{V}_i & = & \tilde{\omega}^T \tilde{K}^\omega \hat{V} \\
\eta^T  \sum_{i=1}^n e_i 1_n^T K_{i}^{\text{droop,I}}\hat{\omega}_{i} &=& \sum_{i=1}^n \eta_i 1_n^T K_{i}^{\text{droop,I}}\hat{\omega}_{i}.
\end{IEEEeqnarray*}
By defining 
\begin{align*}
\bar{V}' = \begin{bmatrix}
\frac{1}{\sqrt{n}} 1_n & S
\end{bmatrix} ^T & \bar{V}
\qquad
\bar{V} = \begin{bmatrix}
\frac{1}{\sqrt{n}} 1_n & S
\end{bmatrix} \bar{V}',
\end{align*} 
we obtain $\bar{V}^T\mathcal{L}_R \bar{V} = \bar{V}''^T S^T\mathcal{L}_R S \bar{V}'' $ and $\bar{V}^T\mathcal{L}_\phi S \bar{\phi} = \bar{V}''^T S^T\mathcal{L}_\phi S \bar{\phi}$, where $\bar{V}'' = [\bar{V}'_2, \dots, \bar{V}'_n]^T$. 
We thus obtain
\begin{align*}
\dot{W} &\le -\begin{bmatrix}
\tilde{\omega} \\ \hat{V}
\end{bmatrix}^T
\underbrace{\begin{bmatrix}
\tilde{K}^\omega(K^V)^{-1}(\tilde{K}^\omega {+} \frac 12 \tilde{K}^\text{droop}) & -\tilde{K}^\omega \\
 -\tilde{K}^\omega & K^V
\end{bmatrix}}_{\triangleq Q_1}
\begin{bmatrix}
\tilde{\omega} \\ \hat{V}
\end{bmatrix} \\
&\;\;\;\; \,-  \begin{bmatrix}
\bar{V}'' \\ {\phi}''
\end{bmatrix}^T 
\underbrace{\begin{bmatrix}
V^\text{nom} S^T\mathcal{L}_R S & -\frac{k_\phi}{2} S^T\mathcal{L}_R S \\
-\frac{k_\phi}{2} S^T\mathcal{L}_R S & \gamma k_\phi S^T\mathcal{L}_R S
\end{bmatrix}}_{\triangleq Q_2}
\begin{bmatrix}
\bar{V}'' \\ {\phi}''
\end{bmatrix} \\
&\;\;\;\; \,- \frac 12 \sum_{i=1}^n \hat{\omega}_i^T K^{\text{droop}}_{i} \hat{\omega}_{i} - \eta^T \mathcal{L}_\eta \eta 
\end{align*}
where  $\tilde{K}^\text{droop}=\diag(\tilde{K}^\text{droop}_{1_1}, \dots, \tilde{K}^\text{droop}_{n_1})$. 
By applying the Schur complement condition for positive definiteness, we see that $Q_1$ is positive definite, since
$
K^\omega (K^V)^{-1}(K^\omega + K^\text{droop}) - K^\omega (K_V)^{-1} K^\omega 
= K^\omega (K^V)^{-1} K^\text{droop} > 0
$. By similar arguments $Q_2$ is positive definite iff
$
( \gamma k_\phi- {k_\phi^2}/{(4 V^\text{nom})} ) S^T \mathcal{L}_R S > 0
$. Clearly the above matrix inequality holds under Assumption~\ref{ass:gamma_coordinated}, since $S^T \mathcal{L}_R S \ge 0$, and $Sx\ne k 1_n$ for $k\ne 0$. 
Thus $\dot{W}\le0$ under Assumption~\ref{ass:gamma_coordinated}, and the set where $W$ is non-decreasing is given by 
$
G = \{ \delta_i'' \in \mathbb{R}^{n_i}, \; i=1, \dots, n, \eta = k 1_n \}
$, 
for any $k\in \mathbb{R}$. The largest invariant set in $G$ with respect to \eqref{eq:hvdc_coordinated_voltage_control_secondary_network_cl_dynamics} is the origin, since $\mathcal{L}^\text{AC}_{i} S_i
\delta_i''  + \frac{K^V_{i}}{K^\omega_{i}} K^\text{droop, I}_{i} 1_{n_i} \eta_i = 0_{n_i}$ implies $\delta_i'' = 0_{n_i}$ and $\eta_i=0$. By LaSalle's theorem, the origin of \eqref{eq:hvdc_coordinated_voltage_control_secondary_network_cl_dynamics} is globally asymptotically stable. 
\end{proof}

\begin{proof}[Proof sketch of Theorem~\ref{th:hvdc_coordinated_voltage_control_pi_link_stability}]
Without loss of generality, let $P^m=0_n$ in \eqref{eq:hvdc_coordinated_voltage_control_pi_link_cl_dynamics}.  
One can verify that 
\begin{IEEEeqnarray*}{lcl}
W &=& \frac{V^\text{nom}}{2} \Bigg( \hat{V}^TC \hat{V} {+} \sum_{i=1}^l  {I}_i^TL^{-1} {I}_i {+}  \sum_{i=1}^{\ell-1}  {V}_i^TC^{\text{line}} {V}_i \Bigg)\\
&& + \frac 12 \bar{\omega}^T K^\omega (K^V)^{-1} M^{-1}\bar{\omega} + \frac 12 \bar{\eta}^T \bar{\eta} + \frac 12 \bar{\phi}^T S^T \mathcal{L}_\phi S \bar{\phi}
\end{IEEEeqnarray*}
is a Lyapunov function  for \eqref{eq:hvdc_coordinated_voltage_control_pi_link_cl_dynamics},
where $C^{\text{line}}=\diag(C^{\text{line}}_1 \dots, C^{\text{line}}_m)$. 
Differentiating $W$ along trajectories of \eqref{eq:hvdc_coordinated_voltage_control_pi_link_cl_dynamics} where $P^m=0_n$ yields 
 $\dot{W}\le 0$. 
By LaSalle's theorem, 
the equilibrium of \eqref{eq:hvdc_coordinated_voltage_control_pi_link_cl_dynamics} is globally asymptotically stable. 
\end{proof}

\bibliography{references}
\bibliographystyle{IEEEtran}

\vphantom{m}  

\end{document}